\documentclass[12pt,a4paper]{article}

\usepackage[top=3cm, bottom=3cm, left = 2.5cm, right = 2.5cm]{geometry} 
\usepackage{amsmath,amssymb}  
\usepackage{amsthm}
\usepackage{graphicx,tikz} 
\usepackage{lipsum}

\usepackage{hyperref}  

\usepackage{xcolor}
\usepackage[all]{xy}
\usepackage{cleveref}
\usepackage{MnSymbol}
\usepackage{float}
\usepackage{proof}
\theoremstyle{definition}       		

\newtheorem{definition}{Definition}[section]
\newtheorem*{definition*}{Definition}
\newtheorem{theorem}{Theorem}[section]
\newtheorem{lemma}{Lemma}[section]

\newtheorem{proposition}{Proposition}[section]
\newtheorem{remark}{Remark}[section]

\newtheorem{example}{Example}[section]

\newcommand{\so}{\rightarrow}

\newcommand{\Lg}{\ensuremath{\mathcal{L} }}

\newcommand{\pair}[1]{\langle #1 \rangle}

\newcommand{\chaves}[1]{\{ #1 \}}

\newcommand{\lfium}{{\it LFI1}}

\newcommand{\luka}{{{\L}ukasiewicz}}

\newcommand{\mls}{\textit{MLMS}}

 \newcommand{\lw}{\ensuremath{L4^w}}
 \newcommand{\ls}{\ensuremath{L4^s}}
 \newcommand{\nw}{\ensuremath{N3^w}}
 \newcommand{\ns}{\ensuremath{N3^s}}
 \newcommand{\bw}{\ensuremath{B3^w}}
 \newcommand{\bs}{\ensuremath{B3^s}}
 \newcommand{\cw}{\ensuremath{C2^w}}
 \newcommand{\cs}{\ensuremath{C2^s}}

\newcommand{\fdetobot}{\ensuremath{FDE^\to_{\bot}}}

\newcommand{\ktto}{\ensuremath{{\it K3}^{\to}}}

\newcommand{\lpto}{\ensuremath{{\it LP}^{\to}_\bot}}

\newcommand{\jt}{\ensuremath{\textit{J3}}}

\newcommand{\lt}{\ensuremath{\textit{\L3}}}

\newcommand{\lj}{\ensuremath{\textit{{\L}J4}}}

\newcommand{\sneg}{\ensuremath{{\sim}}}

\label{xLOGICS}{}

\newcommand{\letkp}{\ensuremath{LET_{K}^+}}

\newcommand{\fde}{\ensuremath{FDE}}
\newcommand{\fdeto}{\ensuremath{FDE^{\to}}}

\newcommand{\bo}{\textsf{b}}
\newcommand{\nei}{\textsf{n}}

\newcommand{\lets}{\textit{LET}s}

\newcommand{\lbb}{\ensuremath{\mathbb{L}}}

\label{xCOLORS} %

\newcommand{\white}{\color{white}}

\label{xSYMBOLS}

\newcommand{\cons}{\ensuremath{{\circ}}}
\newcommand{\con}{\ensuremath{{\circ}}}

\newcommand{\defi}{\stackrel{\text{\tiny def}}{=} }

\label{xGENERAL COMMANDS}%

\newcommand{\mh}{\noindent}

\newcommand{\m}{\vspace{1mm}}
\newcommand{\mm}{\vspace{2mm}}
\newcommand{\mmm}{\vspace{3mm}}

\newcommand{\setl}{\setlength\itemsep{-0.2em}}

\newcommand{\bqu}{\begin{quote}} 
\newcommand{\equ}{\end{quote}}
\newcommand{\enr}{\begin{enumerate}[label={(\arabic*)}, resume]}
\newcommand{\eenr}{\end{enumerate}}

\label{xPORTUGUES}

\usepackage{multicol}
\usepackage[authoryear]{natbib}
\setcitestyle{round}

\usepackage{lineno}
\usepackage{pifont}

\def \SubLat {\mathsf{DsubLat}}
\def \LAT {\mathsf{LAT}}
\def \mls {\textsf{MLMS}}

\begin{document}
 
\title{\bf On Many-logic modal structures and information-based logics\footnote{This article is the preprint version of an article (doi number: 10.1007/s10849-026-09467-x) accepted to publication in \textit{The Journal of Logic Language and Information - Springer}.}} 

\author{M. Martins$^1$,  A. Rodrigues$^2$, M. Coniglio$^3$, 
and A. Freire$^4$ \\ 
[4mm]
$^1$ CIDMA, University of Aveiro,   
{martins@ua.pt} \\
$^2$ Federal University of Minas Gerais, abilio.rodrigues@gmail.com \\ 
$^3$ University of Campinas,
coniglio@unicamp.br\\
$^4$ University of Brasília, 
{alfrfreire@gmail.com}}


\maketitle
\begin{abstract} This paper proposes an approach to information-based logics using \textit{many-logic modal structures} (\mls).
These structures can express accessibility relations between worlds with different underlying logics by anchoring them to a    base lattice, which contains the semantics of each logic as a down-complete sublattice. 
  \mls \,  are suitable for representing  connections between information states (i.e., configurations of databases) and the evolution of information states over time. We will illustrate the application of \mls \, by means of the  six-valued logic of evidence and truth \letkp, related to the lattice $L6$, and some four-, three-, and two-valued logics related to 
  down-complete sublattices of $L6$. These logics are capable of representing paracomplete,  paraconsistent, and classical contexts with six-, four-,  three-, and two-valued scenarios.  
\end{abstract}

\section{Introduction}

In 1977 Belnap published a very influential paper with the suggestive title `How a computer should think' \citep{belnap.1977.how}. His motivation was to provide an account of logical consequence appropriate for a computer dealing with possibly inconsistent and incomplete information. 
What Belnap introduced, however, was not really the formal system, since it already existed in the form of the implication-free fragment of the relevant logic \textit{EQ}, 
which is the well-known logic of first-degree entailment (\textit{FDE}) \citep{and.bel.1963}.  
Nor was it a four-valued semantics for \textit{FDE}, since such a semantics and the corresponding lattice had been investigated by  \cite{dunn.1966}.   
What is new in Belnap's paper is the intuitive interpretation of \textit{FDE} as an \textit{information-based logic}, which is a logic designed  to be used by a computer conceived  as a `question-answering
 system' that answers questions based on deductions that take what is stored in its database
 as premises. 
 The databases envisaged by
 Belnap may contain contradictory information or lack of information about a certain topic and  the computer should be   able to deal with these situations in a sensible way. Thus, 
  its underlying logic   must be both paraconsistent and paracomplete, i.e. neither explosion nor excluded middle holds. 

The computer conceived by Belnap `modifies its current set-up' when it gets new information. 
A `set-up', thus, is a specific configuration of the computer over time, 
and these configurations can be seen as \textit{\it information states} 
represented by worlds in a Kripke-style model.
 In this setting, an information state  is the set of sentences that hold within the corresponding world.  
The notion of possibly incomplete but consistent information states 
appears in  \cite{kripke.1965} and independently in  \cite{grzegorczyk.1964}. 
In Wansing \cite{wans.93}   the notion of possibly inconsistent and incomplete information states is applied to information-based logics.

Let us now consider Kripke models where information states are subjected to different logics. 
In this setting: (i) how can information be `transferred' from one world to another? and 
(ii) how can a modal operator $\Box$ be defined when worlds have different underlying logics?
To answer these questions, we introduce here an approach to information-based logics extending the concept of
many-logic modal structures (\mls),  a framework for Kripke semantics proposed 
by  \cite{RM24}. 
Rather than considering sublattices of a base lattice, as done by Freire and Martins, 
we consider here down-complete sublattices (dc-sublattices for short), which allow us to incorporate a broader class of formal systems related to the base lattice (see~\cite[]{Freire_Madeira_Martins}).

\mls s  are well-suited for representing the connections between information states 
viewed as  databases with different underlying logics,   the evolution of databases over time, and 
databases accessed by users with different privileges.  
This framework integrates worlds with different formal systems related to dc-sublattices of a given base lattice. 
These modal structures can express connections between these formal  systems -- 
and thus between worlds in the corresponding many-logic structure -- by anchoring them to the base  lattice 
 that contains the semantics of each formal system as a dc-sublattice. 
In this context, the  notions of necessity and possibility are redefined in a non-standard way.

In addition to introducing the general formal machinery of \mls s, 
we will explore some of their applications based on  dc-sublattices of the lattice  $L6$,
defined by the six-valued semantics of the logic of evidence and truth \letkp\ \citep{con.rod.2023}. 
Logics of evidence and truth (\textit{LETs}) are equipped with a classicality operator \con\ that 
recovers classical logic for sentences in its scope. Thus, when $\con A$ holds, 
the sentence $A$ is subjected to classical logic, 
otherwise the underlying logic of $A$ is \fde, or an extension of \fde.   
\lets\ can   be interpreted in terms of information, and in this case a 
sentence $\con A$ means that the information conveyed by $A$, positive or negative, is reliable. 
The connective \con, thus,  can be thought of as a kind of \textit{certification}.   
$L6$ extends the well-known logical lattice $L4$,
defined by the four-valued semantics of \textit{FDE} 
by adding a new top and a new bottom,
which correspond to two semantic values intuitively interpreted as reliable positive information 
and reliable negative information \cite[Sect.~4.1]{con.rod.2023}.

Among the dc-sublattices of $L6$, we find lattices defined by the semantics of
the logics \fdetobot, \ktto, \lpto, which are extensions, respectively, of 
\fde, Kleene's \textit{K3}, and the logic of paradox \textit{LP}  
with a material implication and a bottom particle.
Dc-sublattices of $L6$ also provide semantics for {\L}ukasiewicz' logic \textit{\L3},
the logic  \textit{J3}, introduced in  
\cite{dottaviano.1970},  
and also for the logic \lj, introduced here. \lj\ is a four-valued paraconsistent and  paracomplete 
 logic obtained by dropping either excluded middle from
\jt\ or explosion from \lt.   
Together with classical logic, the matrix logics mentioned above define the set \lbb\ that will be used here to exemplify  the many-logic modal structures. 

The remainder of this paper is structured as follows.  
 Section \ref{sec.mlms} presents the general machinery of many-logic modal structures. 
 Section \ref{sec.from4to6} introduces the logic \letkp\ and the lattice $L6$ with the corresponding six 
 scenarios.  
In this Section we  also present the dc-sublattices of $L6$,  the respective matrix logics, with four, three and two values, 
and the intuitive interpretation of them in terms of subsets of the six scenarios of \letkp. 
 In  Section~\ref{sec.twist.and.semantics} we show how the multivalued semantics of the logics mentioned in   Section \ref{sec.from4to6} are obtained by means of twist structures based on non-deterministic 
two-valued semantics. 
  Section~\ref{sec.examples} provides examples of \mls s based on $L6$ and its dc-sublattices, and    
 in Section~\ref{sec.modalities} we  consider how the topics discussed in this paper align with standard validities in modal logic.
  We finish this paper with some remarks on future works in Section~\ref{sec.final}.



\label{aa}

\section{Many-logic modal structures}   \label{sec.mlms}


We start by introducing some notions related to lattices that will be important in the definition of the main concept fo this section --  many-logic modal structures.

\begin{definition} Let  $\mathsf{L}=\langle L,\cdot,{+}\rangle$ be a complete lattice with associated partial order $\leq$. A {\em down-complete sublattice (dc-sublattice)}  of  $\mathsf{L}$ is a complete lattice $\mathsf{L}'=\langle L',\cdot',{+}'\rangle$  with associated partial order $\leq'$ such that:
\begin{enumerate}
	\item $L' \subseteq L$;
	\item ${\leq}'={\leq} \cap (L' \times L')$, i.e.: for every $x,y \in L'$, $x \leq' y$ iff $x \leq y$.
\end{enumerate}

We denote by   $\SubLat(\pair{{L},\cdot,+})$  the set of all \textit{dc-sublattices} of $\pair{{L},\cdot,+}$. 
\end{definition}

\begin{definition}[The down-interpretation]
		 Let  $\mathsf{L}=\langle L,\cdot,{+}\rangle$ be a complete lattice with associated partial order $\leq$, and let  $\mathsf{L}'=\langle L',\cdot',{+}'\rangle$ be a  dc-sublattice  of  $\mathsf{L}$. The {\em down-interpretation} of value $x \in L$ in $\mathsf{L}'$ is defined as follows:
$$x^{\mathsf{L}'} = \bigvee_{\mathsf{L}'} \{y \in L' \ : \ y \leq x\}.$$
\end{definition}

\noindent
Observe that, if $x \in L'$ then $x^{\mathsf{L}'} =x$. In addition, if $\{y \in L'  \ : \   y \leq x \} = \emptyset$, then $x^{\mathsf{L}'} $ is the least value in $L'$. 
Moreover, given a complete lattice $\mathsf{L}=\langle L,\cdot,{+}\rangle$ and  a  dc-sublattice  $\mathsf{L}'=\langle L',\cdot',{+}'\rangle$ of  $\mathsf{L}$, with associated partial orders $\leq$ and $\leq'$, resp., we have that  
  $x \leq y$ implies  $x^{\mathsf{L}'} \leq' y^{\mathsf{L}'}$, for any $x, y \in L$.
It is not difficult to see that this property implies, for every $\emptyset \neq X \subseteq L'$, 
$$\bigvee_{\mathsf{L}'}X = \big(\bigvee_{\mathsf{L}} X\big)^{\mathsf{L}'} \ \ \mbox{ and } \ \ \bigwedge_{\mathsf{L}'}X = \big(\bigwedge_{\mathsf{L}} X\big)^{\mathsf{L}'}.$$


If $X \subseteq L$ then $X^{\mathsf{L}'}$ denotes the set $\{ x^{\mathsf{L}'} \ : \ x \in X\}$.

\begin{proposition}\label{Prop3}  Let  $\mathsf{L}=\langle L,\cdot,{+}\rangle$ be a complete lattice and let  $\mathsf{L}'=\langle L',\cdot',{+}'\rangle$ be a  down-complete sublattice  of  $\mathsf{L}$. Then for every $\emptyset \neq X \subseteq L$,
$$\bigvee_{\mathsf{L}'}X^{\mathsf{L}'} \leq' \big(\bigvee_{\mathsf{L}} X\big)^{\mathsf{L}'} \ \ \mbox{ and } \ \ \bigwedge_{\mathsf{L}'}X^{\mathsf{L}'} \geq' \big(\bigwedge_{\mathsf{L}} X\big)^{\mathsf{L}'}.$$
\end{proposition}

\begin{proof} 
(1) Let $u=\bigvee_{\mathsf{L}} X$ and $z=u^{\mathsf{L}'}$. If $x \in X$ then $x \leq u$ and so, $x^{\mathsf{L}'} \leq' u^{\mathsf{L}'}=z$. From this $\bigvee_{\mathsf{L}'}X^{\mathsf{L}'} \leq' z$.\\ 
(2)  Let $w=\bigwedge_{\mathsf{L}} X$ and $z'=w^{\mathsf{L}'}$.
Let $x' \in L'$ such that $x' \leq w$. If $x \in X$ then $w \leq x$ and so $x' \leq x$. Clearly, $x'=(x')^{\mathsf{L}'} \leq' x^{\mathsf{L}'}$, and so $z'=\bigvee_{\mathsf{L}'}\{x' \in L' \ : \ x' \leq w\} \leq' x^{\mathsf{L}'}$, for every $x \in X$. From this $z' \leq' \bigwedge_{\mathsf{L}'} X^{\mathsf{L}'}$.
\end{proof}

%
%

\begin{definition} Let $\mathsf{L}=\langle L,\cdot,{+}\rangle$ be a complete lattice. We say that $L$ is a $[\land\bigvee]$-lattice if, for every $X \subseteq L$ and every $x \in L$,
$$x \cdot \bigvee_\mathsf{L} X =  \bigvee_\mathsf{L} \{x \cdot y \ : \ y \in X\}.$$
\end{definition}

\begin{lemma} \label{Lemma1}  Let  $\mathsf{L}=\langle L,\cdot,{+}\rangle$ be a complete lattice and let  $\mathsf{L}'=\langle L',\cdot',{+}'\rangle$ be a  dc-sublattice  of  $\mathsf{L}$. If $x,y \in L'$ then $x \cdot' y \leq x \cdot y$ and  $x + y \leq x +' y$.\end{lemma}

\begin{proof}
	
Let $Lb_{\mathsf{L}'}(\{x,y\})=\{x' \in L' \ : \ x' \leq' x \mbox{ and } x' \leq' y\}$ and $Lb_{\mathsf{L}}(\{x,y\})=\{z \in L \ : \ z \leq x  \mbox{ and }  z \leq y\}$. Since $Lb_{\mathsf{L}'}(\{x,y\}) \subseteq Lb_{\mathsf{L}}(\{x,y\})$ then
$$x \cdot' y = Max_{\mathsf{L}'}(Lb_{\mathsf{L}'}(\{x,y\})) = Max_{\mathsf{L}}(Lb_{\mathsf{L}'}(\{x,y\})) \leq  Max_{\mathsf{L}}(Lb_{\mathsf{L}}(\{x,y\})) = x \cdot y.$$
The proof that  $x + y \leq x +' y$ is analogous, but now by using the fact that, if $\emptyset \neq X \subseteq Y \subseteq L$ then $Min_{\mathsf{L}}(X) \geq Min_{\mathsf{L}}(Y)$.
\end{proof}

\begin{proposition}\label{Prop4} Let  $\mathsf{L}=\langle L,\cdot,{+}\rangle$ be a complete lattice and  $\mathsf{L}'=\langle L',\cdot',{+}'\rangle$ be a  dc-sublattice  of  $\mathsf{L}$ that is a  $[\land\bigvee]$-lattice. Then, for every $x,y \in L$: $x^{L'} \cdot' y^{L'} = (x \cdot y)^{L'}$.\end{proposition}

\begin{proof}
Let  $x,y \in L$. By Proposition~\ref{Prop3} , it is enough to prove that $x^{L'} \cdot' y^{L'} \leq (x \cdot y)^{L'}$. Using the fact that $\mathsf{L}'$ is a $[\land\bigvee]$-lattice, it follows that 
$$x^{L'} \cdot' y^{L'} = \bigvee_{\mathsf{L}'}X_{x,y}$$
where $X_{x,y}=\{x' \cdot' y' \ : \ x',y' \in L', \ x' \leq x, \mbox{ and } y' \leq y\}$. Now, let $x' \cdot' y' \in X_{x,y}$. Since $x' \leq x$ and $y' \leq y$ then, by Lemma~\ref{Lemma1}, $x' \cdot' y' \leq x' \cdot y' \leq x \cdot y$. This shows that $X_{x,y} \subseteq Y_{x,y}=\{z' \in L' \ : \ z' \leq x \cdot y\}$. From this, $x^{L'} \cdot' y^{L'} = \bigvee_{\mathsf{L}'}X_{x,y} \leq' \bigvee_{\mathsf{L}'}Y_{x,y} = (x \cdot y)^{L'}$ as required.
\end{proof}

\

%
%

\begin{definition} (Filtered $\Lg$-lattices)

	\m\mh Let a \textit{filtered $\Lg$-lattice} be a matrix logic  $ \pair{{\bf L},{\rm D}} $ 
	such that    $\textbf{L}$ is an algebra over $ \Sigma $ with domain $L$ and  $ {\rm D} \subseteq  L $, 
	where the $ \Sigma_L $-reduct  $  \pair{L,\cdot, +} $ is a lattice.
\end{definition}

Consider a \textit{language $\Lg$},  defined over  a signature  $ \Sigma $ containing $ \Sigma_L =  \pair{\land, \lor}$, 
and a set of sentential letters~$ Var $.

\begin{definition} [Many-logic modal  structures, \mls]\label{def.mlms}

\m\mh  Let  $\LAT \subseteq \SubLat(\pair{{L},\cdot ,+})$, such that $ \pair{L,\cdot, +} \in \LAT$,   and $\Lg^\Box$ the language obtained by extending $\Lg$ with 
 the unary connective $\Box$.  
   A \textit{many-logic modal structure} $M$ over $\LAT$ and $\Lg^\Box$  is a tuple $\pair{W, R, I,v}$ such that 
$W \neq \emptyset$ and:

	\begin{enumerate} \setl
    \item $\mathbb{L}$ is a set of matrix logics $ \pair{{\bf L}',{\rm D}'} $ over $\Sigma$ such that $ \Sigma_L $-reduct $\pair{{L}',\cdot,+}$ of ${\bf L}'$ belongs to $\LAT $.\footnote{In \cite{RM24} it is only allowed matrix logics over lattices in $\LAT $ induced by the filter $D$ of ${\bf L}$. Namely, for each  ${\bf L}'$ belongs to $\LAT $ we have the matrix logic  $ \pair{{\bf L}', D\cap L'}$. Here we have a more general notion of many-logic modal structure, than can be worth for applications.  Moreover, in the original work we just consider sublattices instead of dc-sublattices. In this context dc-sublattices were already considered in \cite{Freire_Madeira_Martins}}
\item $I:W\to \mathbb{L}$ assigns a matrix logic $ \pair{{\bf L}_w,{\rm D}_w} $ to each world $w \in W$. By simplicity, and when there is no risk of confusion, we may write $I(w)=L_w$, where $L_w$ is the domain of  ${\bf L}_w$.\footnote{Observe that it is possible to have an \mls\, and some $w\neq w'$ such that 
$I(w) \neq I(w')$ but $L_w=L_{w'}$, that is, two different logics with the same domain. In this case,  the notation $I(w)=L_w$ and $I(w')=L_{w'}$ will not be allowed. In the structures investigated here, however, each lattice is the domain of exactly one logic.} 
\item $R$ is a relation from worlds to worlds, i.e. $R \subseteq W \times W$.
\item   $v$ is a \textit{valuation function} 
 from the set $Var$ of sentential letters to values in $I(w)$, i.e. for each world $w \in W$ and $p\in Var$,  
  $v: W \times Var \longrightarrow L_w$ such that $v(w,p)\in L_w$.

    \end{enumerate}

	\end{definition}	


%
%
%
%


In each world $w$ a valuation $v_w$ 
is extended to all formulas of $\Lg$ based on the  matrix logic $\pair{\textbf{L}_w, {\rm D}_w}$ related to the lattice 
 assigned to $w$, given by $I(w)$. 
  To extend a valuation $v_w$  to all formulas of 
$\Lg^\Box$ (i.e.~to assign semantic values to formulas $\Box A$) we interpret in a dc-sublattice 
$L' \subseteq L$ elements not necessarily in $L'$ according to the down interpretation.
 
 Let $w$ and $w'$ be worlds,  $L_w$ and $L_{w'}$ the corresponding lattices,    
 and $v_{w}(A)$ and $v_{w'}(A)$ the values assigned to a formula $A$ in $w$ and $w'$ by a valuation 
 $v$. We denote by  $ (v_{w}(A))^{L_{w'}}$ and $ (v_{w'}(A))^{L_{w}}$ the values $v_{w}(A)$ and $v_{w'}(A)$  `seen' from the perspective of the worlds $w'$ and $w$ respectively, according to the down-interpretation.

The semantic value of a  formula $\Box A$ in a world $w$ 
is the infimum of the set of  values of $A$ in all the worlds accessed from $w$ as seen 
from the perspective of the local lattice $L_w$.  
A  clause for the connective $\Box$ is thus  defined as follows:  

\begin{definition} [Semantic clause for $\Box$] \label{def.clause.box} \quad \\
\begin{center}
($\ast$) $v_{{w}}(\Box  A)=\bigwedge\limits_{L_{{w}}}\{(v_{w'}(A))^{L_{{w}}}  \ : \  w' \in W \land {w} R w'  \}$. 
\label{clause.box}
\end{center}
\end{definition}

\mh We say that $A$ holds in $w$ in a structure $M$, denoted by $M, w \vDash A$, if $v_w(A) \in \rm D$. 
However, in the framework adopted here, saying that a formula has a specific value  ($v_w(A)=x$)  in a model $M$ and a world $w$   is  generally  
more informative than  saying that it holds  (or does not hold). 
This is due to the semantics of the multivalued logics we will explore below, which may involve multiple designated (or non-designated) values.

\begin{example}

Consider the language $\Lg$ defined over the signature $\Sigma = \{  \neg, \lor, \land \}  $, which is 
the language of \fde, and the lattice $\lw$ below\footnote{Please do not confuse \lw\ with the notation $L_w$ introduced above.}, 
the well-known logical order defined by the four-valued  semantics 
of \fde:  
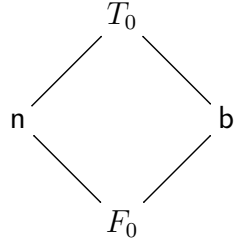
\begin{center}
$\xymatrix{
&{T_0} \ar@{-}[dl]\ar@{-}[dr] &&\\
\nei \ar@{-}[dr]   & & \bo \ar@{-}[dl]  \\
& F_0 &
}$
\end{center}

\mh 
The matrix logic $\pair{L4^w, \{ T_0, \bo\}}$ over \Lg\  is the logic \fde.  
Consider the lattices  \(\bw = \{T_0, \bo, F_0\}\) and \(\nw = \{T_0, \nei, F_0\}\) 
and the matrix logics $\pair{\bw , \{ T_0, \bo\}}$ and $\pair{\nw , \{ T_0 \}}$, 
 which are known to define the logic of paradox \textit{LP} and Kleene's \textit{K3}, 
 respectively -- the reason we use  $ T_0 $ and $F_0$ instead of $ T$ and $F$ and write $\lw$  
 instead of $L4$ will become clear soon. The language $\Lg^\Box$ extends \Lg\ with $\Box$.   
 Now let $M$ be a \mls\ such that $R$ is an equivalence relation and:

\mm  \quad     \(\LAT =\{L4^w, B3^w,N3^w\}\); 

\quad        $W=\{w_1,w_2,w_3\}$;
       
      \quad   $L_{w_1}= L4^w, \ L_{w_2}= B3^w, \ L_{w_3}= \nw$;  

      \quad $v_{w_1}(p)=\bo$, 
       $v_{w_2}(p)=T_0$, 
      $v_{w_3}(p)=T_0$. 

 
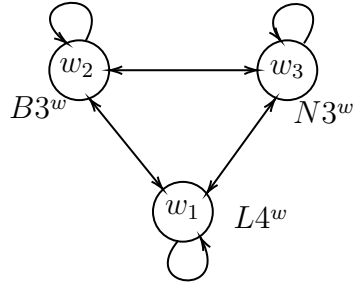
\begin{figure}[H]
     \centering
    \tikzset{every picture/.style={line width=0.75pt}} 

\begin{tikzpicture}[x=0.75pt,y=0.75pt,yscale=-0.6,xscale=0.6]

\draw   (203,86) .. controls (203,72.19) and (214.19,61) .. (228,61) .. controls (241.81,61) and (253,72.19) .. (253,86) .. controls (253,99.81) and (241.81,111) .. (228,111) .. controls (214.19,111) and (203,99.81) .. (203,86) -- cycle ;
\draw   (377,86) .. controls (377,72.19) and (388.19,61) .. (402,61) .. controls (415.81,61) and (427,72.19) .. (427,86) .. controls (427,99.81) and (415.81,111) .. (402,111) .. controls (388.19,111) and (377,99.81) .. (377,86) -- cycle ;
\draw   (289.67,204.33) .. controls (289.67,190.53) and (300.86,179.33) .. (314.67,179.33) .. controls (328.47,179.33) and (339.67,190.53) .. (339.67,204.33) .. controls (339.67,218.14) and (328.47,229.33) .. (314.67,229.33) .. controls (300.86,229.33) and (289.67,218.14) .. (289.67,204.33) -- cycle ;
\draw    (237.93,111.55) -- (296.68,184.01) ;
\draw [shift={(297.94,185.56)}, rotate = 230.96] [color={rgb, 255:red, 0; green, 0; blue, 0 }  ][line width=0.75]    (10.93,-3.29) .. controls (6.95,-1.4) and (3.31,-0.3) .. (0,0) .. controls (3.31,0.3) and (6.95,1.4) .. (10.93,3.29)   ;
\draw [shift={(236.67,110)}, rotate = 50.96] [color={rgb, 255:red, 0; green, 0; blue, 0 }  ][line width=0.75]    (10.93,-3.29) .. controls (6.95,-1.4) and (3.31,-0.3) .. (0,0) .. controls (3.31,0.3) and (6.95,1.4) .. (10.93,3.29)   ;
\draw    (255,86) -- (375,86) ;
\draw [shift={(377,86)}, rotate = 180] [color={rgb, 255:red, 0; green, 0; blue, 0 }  ][line width=0.75]    (10.93,-3.29) .. controls (6.95,-1.4) and (3.31,-0.3) .. (0,0) .. controls (3.31,0.3) and (6.95,1.4) .. (10.93,3.29)   ;
\draw [shift={(253,86)}, rotate = 0] [color={rgb, 255:red, 0; green, 0; blue, 0 }  ][line width=0.75]    (10.93,-3.29) .. controls (6.95,-1.4) and (3.31,-0.3) .. (0,0) .. controls (3.31,0.3) and (6.95,1.4) .. (10.93,3.29)   ;
\draw    (390.64,111.66) -- (335.78,187.28) ;
\draw [shift={(334.6,188.9)}, rotate = 305.96] [color={rgb, 255:red, 0; green, 0; blue, 0 }  ][line width=0.75]    (10.93,-3.29) .. controls (6.95,-1.4) and (3.31,-0.3) .. (0,0) .. controls (3.31,0.3) and (6.95,1.4) .. (10.93,3.29)   ;
\draw [shift={(391.82,110.05)}, rotate = 125.96] [color={rgb, 255:red, 0; green, 0; blue, 0 }  ][line width=0.75]    (10.93,-3.29) .. controls (6.95,-1.4) and (3.31,-0.3) .. (0,0) .. controls (3.31,0.3) and (6.95,1.4) .. (10.93,3.29)   ;
\draw    (310.56,228.75) .. controls (276.92,269.08) and (355.64,276.67) .. (330.17,227.45) ;
\draw [shift={(329.36,225.95)}, rotate = 61.28] [color={rgb, 255:red, 0; green, 0; blue, 0 }  ][line width=0.75]    (10.93,-3.29) .. controls (6.95,-1.4) and (3.31,-0.3) .. (0,0) .. controls (3.31,0.3) and (6.95,1.4) .. (10.93,3.29)   ;
\draw    (233.36,61.45) .. controls (260.3,11.98) and (171.71,26.52) .. (216.37,62.36) ;
\draw [shift={(217.76,63.45)}, rotate = 217.44] [color={rgb, 255:red, 0; green, 0; blue, 0 }  ][line width=0.75]    (10.93,-3.29) .. controls (6.95,-1.4) and (3.31,-0.3) .. (0,0) .. controls (3.31,0.3) and (6.95,1.4) .. (10.93,3.29)   ;
\draw    (415.91,63.48) .. controls (447.59,15.96) and (369.5,24.64) .. (395.48,59.98) ;
\draw [shift={(396.3,61.05)}, rotate = 231.92] [color={rgb, 255:red, 0; green, 0; blue, 0 }  ][line width=0.75]    (10.93,-3.29) .. controls (6.95,-1.4) and (3.31,-0.3) .. (0,0) .. controls (3.31,0.3) and (6.95,1.4) .. (10.93,3.29)   ;

\draw (432.67,209.18) node   [align=left] {\begin{minipage}[lt]{68pt}\setlength\topsep{0pt}
$L4^w$
\end{minipage}};
\draw (404.33,109.18) node [anchor=north west][inner sep=0.75pt]   [align=left] {$N3^w$};
\draw (245,117.18) node   [align=left] {\begin{minipage}[lt]{68pt}\setlength\topsep{0pt}
$B3^w$
\end{minipage}};
\draw (297.33,192.18) node [anchor=north west][inner sep=0.75pt]   [align=left] {$w_1$};
\draw (385,74.85) node [anchor=north west][inner sep=0.75pt]   [align=left] {$w_3$};
\draw (209,73.85) node [anchor=north west][inner sep=0.75pt]   [align=left] {$w_2$};

\end{tikzpicture}
 \end{figure}

\mh The model $M$ can be thought of as representing three databases connected to each other. The underlying logic of $w_1$
 is \fde, $w_2$ is a closed-world database that admits contradictory information  
 and $w_3$ an open-world database that does not admit inconsistencies.

 First, notice how each world `sees' the others. $w_1$ is an $L4^w$-world, so it sees $w_2$ and $w_3$ 
 `exactly as they are', since the lattices
 $B3^w$ and $N3^w$ are contained in $L4^w$. However,   the converse is not true. 
 According to the down interpretation, when a $B3^w$-world accesses an $L4^w$-world, the value 
 $\nei$ becomes $F_0$, and when a $N3^w$-world accesses a $L4^w$-world, the value 
 $\bo$ becomes $F_0$, that is: $(\bo )^{N3}=F_0$ and $(\nei )^{B3}=F_0$. We have thus the following values in the model $M$ above: 


\begin{multicols}{3}
    
$v_{w_1}(p)=\bo $

$(v_{w_2}(p))^{L_{w_1}}=T_0$
 
$(v_{w_3}(p))^{L_{w_1}}=T_0$

$v_{w_2}(p)=T_0$

$(v_{w_1}(p))^{L_{w_2}}=\bo $
 
$(v_{w_3}(p))^{L_{w_2}}=T_0$

$v_{w_3}(p)=T_0$

$(v_{w_1}(p))^{L_{w_3}}=F_0$
 
$(v_{w_2}(p))^{L_{w_3}}=T_0$

\end{multicols}

\mh Second, the value of $\Box p$ in a world $w$ is given by the infimum of the values of $p$ in the worlds accessed from $w$,  
as seen from $w$. Therefore, 

\begin{multicols}{3}

$v_{w_1}(\Box p)=\bo $

$v_{w_2}(\Box p)=\bo $

$v_{w_3}(\Box p)=F_0$
    
\end{multicols}
    
\end{example}

\m According to the intuitive interpretation we propose, 
 $v_w(\Box A)=x$  means 
`$A$ has the semantic value $x$ for an agent accessing the information available from $w$', 
which is the information in the worlds accessible from $w$.

 \section{From four to three and to six scenarios} \label{sec.from4to6}

We start by revisiting the four-valued information-based interpretation of \fde. 
 We will also see how to provide information-based 
 interpretations to paraconsistent and paracomplete three-valued logics,  
 and the rationale of extending the four scenarios of \fde\ to the six scenarios of the logic  \letkp. 

 The  four-valued semantics for \fde, with the  values 
  $T_0, F_0, \bo, $ and $ \nei $, has the following intuitive meaning:

\begin{itemize}\setl  
\item[(i)] $v(A) = T_0$: only positive information $A$  ($A$ holds,  $\neg A$ does not hold); 
 \item[(ii)]  $v(A) = F_0$: only negative information $A$ ($\neg A$ holds,  $A$ does not hold); 
  \item[(iii)]  $v(A) = \bo$: contradictory information about $A$ (both $A$ and $\neg A$ hold);  
  \item[(iv)]  $v(A) = \nei$: no information at all about $A$ (neither $A$ nor $\neg A$ holds).
\end{itemize}

\mh   \citet[p.~38]{belnap.1977.how} explains the values represented here by $T_0$ and  $F_0$ as `told true' and 
  `told false' signs,  in the sense that a computer `has been told' that $A$  is true and $A$  is false respectively --  
    meaning   that the computer has received, respectively, positive information $A$ and negative information $A$.

 \subsection{The logic \letkp\ and the lattice $L6$}

  Logics of evidence and truth (\lets) are a family of paracomplete and paraconsistent logics    
that extend \fde\  with a classicality operator \con\ that recovers classical negation for sentences in its scope 
as follows:

\begin{itemize}\setl 
        \item[-]  For some $ A $ and $ B $, $ A ,\neg A \nvdash B $, while for every $ A $ and $ B $, $\circ A , A ,\neg A \vdash B $. 
        
       \item[-]  For some $ A $, $\nvdash A  \lor \neg A $, while for every $A$,  
       $\con A  \vdash  A  \lor \neg A $.\footnote{For a standard disjunction $\lor$.} \label{prop.lfi.lfu}
       \end{itemize}

 \mh  
 The   intuitive interpretation of  \lets\ can be given in terms of 
 evidence 
or in terms of information. In the latter case, 
the intuitive reading of a formula $\con A$ is that the information conveyed by $A$, positive or negative, is reliable, and  when $\con A$ does not hold, it means that there is no reliable information about $A$. 
  It is  assumed that reliable information is subject to classical logic, 
  while the underlying logic of non-reliable information 
 is \fde\ or an extension of \fde.\footnote{For the origins and motivations of \lets, and the interpretation in terms of 
 evidence and information, see \cite{rod.ant.lu,rod.car.llp}.  }    
  In this setting, the connective \con\ can be thought of as  kind of \textit{certification}, 
 in the sense that a sentence $\con A$ means that the information $A$, positive or negative, 
 has been  marked as reliable, i.e.~{\it certified} by some  user with such privileges.  
 Thus,  when $\con A$ holds, we have either $\con A \land A$ or 
   $\con A \land\neg A$ but not both, which means, respectively, that $A$ has been certified or $\neg A$ has been  
 certified. 

 \lets\  extend the four scenarios of \fde\   in a natural way, 
 by adding two more scenarios to the four scenarios mentioned above, viz.,  reliable positive and reliable negative information.
 
 \begin{itemize}\setl  
\item[] When $ \con A $ does not hold:

\begin{itemize}

\item[(i)] $v(A) = T_0$: only positive information $A$  ($A$ holds,  $\neg A$ does not hold); 
 \item[(ii)]  $v(A) = F_0$: only negative information $A$ ($\neg A$ holds,  $A$ does not hold); 
  \item[(iii)]  $v(A) = \bo$: contradictory information about $A$ (both $A$ and $\neg A$ hold);  
  \item[(iv)]  $v(A) = \nei$: no information at all about $A$ (neither $A$ nor $\neg A$ holds).
\end{itemize}

\item[] When $ \con A $  holds:
\begin{itemize}

\item[(v)] $v(A) = T$: reliable (certified) positive information $A$;   
 \item[(vi)]  $v(A) = F$: reliable (certified) negative information $A$.
\end{itemize}
  
\end{itemize}

\mh The semantic values $T$ and $F$ above intend to express that the respective information 
has been marked as reliable, while the values $T_0$ and $F_0$ indicate that the respective information  
has not been marked as reliable. 
 We may also read $T_0$ and $F_0$ as \textit{just true} and \textit{just false} respectively, and 
 $T$ and $F$ as \textit{certified true} and \textit{certified false} respectively.

The logic \letkp\ is a logic of evidence and truth introduced in \cite{con.rod.2023}.  
  The \con-free fragment of \letkp\ is the logic \fdeto, 
 which extends \fde\ with a material implication.\footnote{As far as we know, \fdeto\ was introduced 
  in \cite{pynko.1999} dubbed $ \mathcal{IDM}4$. It appears in  \cite{hazen.2019}  
  under the name \fdeto, which we adopt here. In \cite{hazen.2019} extensions of \textit{K3} and \textit{LP} with a material 
  implication are investigated.}  
    Like \fde, \fdeto\ admits a lattice-based 
 semantics with the semantic values $ T_0, \nei, \bo, F_0$.  
 \letkp\   admits a lattice-based semantics with the six values $ T, T_0, \bo, \nei , F_0, F $, 
 where $ T, T_0, \bo$ are  designated (see Section \ref{sec.twist.and.semantics}). Below, the lattice  $L6$, defined by the semantics 
 of \letkp\ \cite[Sect.~4.1]{con.rod.2023}:  

$$
\xymatrix{
&T\ar@{-}[d]&\\
& T_0\ar@{-}[dl]\ar@{-}[dr] &\\
\nei \ar@{-}[dr] & & \bo\ar@{-}[dl]\\
& F_0\ar@{-}[d]\\
&F&
}$$


 \subsection{The dc-sublattices of $L6$ and their associated logics}  \label{sec.sublattices}

Now consider the following extensions of the logic \fdeto, each of which admits a   semantics based 
on dc-sublattices of $L6$ and whose designated values are contained in \( \chaves{ T, T_0, \bo} \):

\begin{description}
\item  \fdetobot, obtained by adding a bottom particle to \fdeto. 
\fdetobot\  is equivalent to the logic $\mathcal{BID}4$  \citep{pynko.1999} and  
admits a lattice-based semantics, with the semantic values $T_0, \nei, \bo, F_0  $. 

The   intuitive interpretation is the same as that of  the   scenarios (i) to (iv) of \fde. 

\item \lpto, obtained by adding excluded middle to \fdetobot. 
\lpto\ can also be defined by adding a material implication and a bottom particle 
to the logic of paradox  \textit{LP}. 
\lpto\ is paraconsistent but not paracomplete, and admits a lattice-based semantics with the semantic values $T_0,  \bo, F_0  $. 

The intuitive interpretation is that of a closed-world database that accepts contradictions but lacks certification information, corresponding to scenarios (i), (ii), and (iii).

\item \ktto, obtained by adding a material implication to  \textit{K3}. Note that bottom can be defined in 
\textit{K3} as $A\land\neg A$. 
\ktto\ is paracomplete but not paraconsistent, and  admits a lattice-based semantics with the semantic values $T_0,  \nei, F_0  $. 

The intuitive   interpretation is similar to that of \textit{K3} as the underlying logic of some 
 implementations of SQL where the third value is
  read `information missing' (see e.g. \citet[pp.~190-1]{Celko2005}).
The corresponding scenarios are  (i), (ii), and (iv). 

\item $CL^w$  is classical logic with the semantic values $T_0$ and $F_0$, corresponding to  scenarios (i) and (ii). 

 \end{description}

\mh The lattices related to the logics above   are the following:

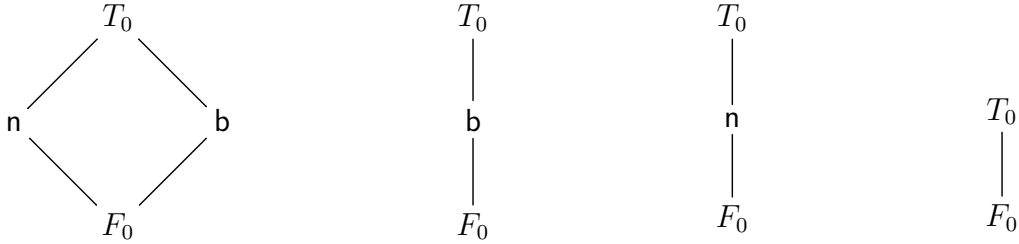
\begin{center}
$\xymatrix{
&{T_0} \ar@{-}[dl]\ar@{-}[dr] &&\\
\nei \ar@{-}[dr]   & & \bo \ar@{-}[dl]  \\
& F_0 &
}$
  \hspace{4mm}
  $
\xymatrix{
&T_0\ar@{-}[d]&\\
& \bo \ar@{-}[d]&\\
&F_0&
}$
  \hspace{4mm}
  $
\xymatrix{
&T_0\ar@{-}[d]&\\
& \nei \ar@{-}[d]&\\
&F_0&
}$
  \hspace{4mm}
$
\xymatrix{ \ \\ \ 
&T_0\ar@{-}[d]&\\
&F_0&
}$

\end{center}
  
\mh Let us call these lattices, respectively, $\lw $, $\bw$,   $\nw$, and $\cw$, 
where the superscript $w$ indicates 
that they are a \textit{weak} interpretation of the corresponding scenarios, in the sense that there is no information 
of reliability being  expressed.    

Now, consider the following logics, which are also extensions of \fdeto\ and admit lattice-based semantics 
that are dc-sublattices of $L6$. Such lattices, however,  instead of the values $T_0$ and $F_0$, contain
the strong values  $T$ and  $F$:

\begin{description} \setl

\item \lj,   a four-valued logic that receives its name because it `combines' (in some sense)  
the paraconsistent logic \jt\  and the 
   \luka' paracomplete  logic \lt.\footnote{In fact, the semantics of all the logics considered here, with the exception of \lj, are associated with sublattices of $L6$. We need the concept of a down-complete sublattice in this paper only to accommodate the logic \lj. Nevertheless, this concept extends the range of logics that can be treated by many-logic modal structures.} In \lj, neither explosion nor excluded middle holds.   
\lj\ admits a four-valued semantic with the values $T, F, \bo, \nei $.\footnote{The logic \lj\ is equivalent to the logic \textit{BS4}, 
investigated by \citet[p.~322]{omori.waragai}. It has been obtained here as a result of changing \lt\ and \jt\ in order to be able 
to express four scenarios. We kept here the name \lj\ to make clear that it has been obtained by `combining', in this sense, \lt\ and \jt. }
\lj\ expresses the scenarios 
(iii), (iv), (v), and (vi). The intuitive interpretation is the same as that of \fde, except that in the scenarios 
with only positive or negative information, such information is considered reliable.

 \item \jt, the three-valued paraconsistent logic introduced by \citep{dottaviano.1970}. 
   \jt\ is equivalent to the logic \lfium\ \citep{carn:marcos:deamo:2000} 
  and admits a three-valued semantics with the  values $T, \bo, F$. 
   
  \jt\ expresses the scenarios (iii),  (v), and (vi), and its intuitive interpretation follows   \citep{carn:marcos:deamo:2000}, where  \textit{LFI1} is  investigated as the underlying logic of closed-world databases that admit inconsistent information.

\item \lt,  the three-valued \luka' logic, which admits a semantics with the values $T,  \nei, F$.  
 
  \lt\ expresses the scenarios  (iv), (v), and (vi), intuitively interpreted like   \ktto, except that reliable information is being expressed.

\item  $CL^s$  is classical logic with the semantic values $T$ and $F$, corresponding to  scenarios (v) and (vi).

\end{description}

\mh 
The lattices related to the logics above are the following:

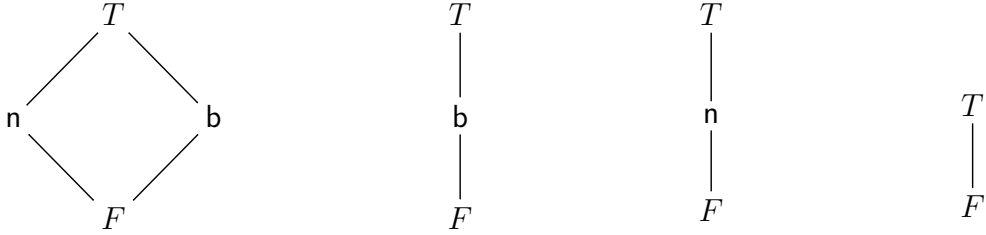
\begin{center}
$\xymatrix{
&{T} \ar@{-}[dl]\ar@{-}[dr] &&\\
\nei \ar@{-}[dr]   & & \bo \ar@{-}[dl]  \\
& F &
}$
  \hspace{4mm}
  $
\xymatrix{
&T\ar@{-}[d]&\\
& \bo \ar@{-}[d]&\\
&F&
}$
  \hspace{4mm}
  $
\xymatrix{
&T\ar@{-}[d]&\\
& \nei \ar@{-}[d]&\\
&F&
}$
  \hspace{4mm}
$
\xymatrix{ \ \\ \ 
&T\ar@{-}[d]&\\
&F&
}$
\end{center}

\mh Let us call these lattices, respectively, $\ls $, $\bs$,   $\ns$, and $\cs$, 
where the superscript $s$ indicates 
that they are a \textit{strong} interpretation of the scenarios represented, in the sense that, except when 
inconsistency or contradiction occur,  the information being expressed is considered reliable.

\mm 
Now make \Lg\  the   language  of   \letkp,     defined over the signature $\Sigma = \{\con, \neg, \land, \lor, \to  \}$. 
 Consider the filtered lattice $\pair{{\bf L},{\rm D}}$  over \Lg\ such that  
  the domain $L$ of   \textbf{L} is the lattice $L6$ 
 and the filter is the set  of designated values of \letkp,   $\{ T, T_0, \bo  \}$. 
 The six-valued semantics of \letkp\ is obtained by means of twist structures
  based on the two-valued  semantics.
All the logics mentioned above, related to dc-sublattices of $ L6 $, 
are {matrix logics} that can be expressed in the language of \letkp. 
The respective four-, three-, and two-valued semantics    
are obtained by means of twist structures just by  adapting the semantics 
of \letkp\ (details will be given in the Section \ref{sec.twist.and.semantics} below). 
 We have thus the following matrix logics:

\begin{description} \setl 
    
 \begin{multicols}{2}  
    \item  \quad $ \letkp = \langle L6, \{ T, T_0, b \}  \rangle $, 

    \item  \quad  $ \fdetobot = \langle \lw, \{ T_0,  b \}  \rangle $

    \item  \quad  $ \lj = \langle \ls, \{ T,   b \}  \rangle $

    \item  \quad   $ \lpto   = \langle \bw, \{ T_0,   b \}  \rangle $

    \item  \quad  $ \jt   = \langle \bs, \{ T,   b \}  \rangle $

    \item 
$ \ktto = \langle \nw, \{ T_0 \}  \rangle $

    \item $ \lt = \langle \ns, \{ T  \}  \rangle $

    \item $ CL^w   = \langle \cw, \{ T_0  \}  \rangle $

    \item  $ CL^s   = \langle \cs, \{ T  \}  \rangle $

    \item  $ \white CL^s   = \langle \cs, \{ T  \}  \rangle $

 \end{multicols}
 
\end{description}
\mh The set $\LAT = \{  L6, L4^w, L4^s, B3^w, B3^s, N3^w, N3^s, C2^w, C2^s \}  $
  is the collection of algebras that underlie the logics above. 


\mm All the logics above can be interpreted in terms of information
 in the sense that they represent different sets of scenarios that are subsets of the six scenarios of \letkp. 
They are thus capable of expressing the deductive behavior of information in databases with specific features, such as closed  and open world databases, 
with or without reliable information, paraconsistent and/or paracomplete, and even classical.

 In the next section, we will   present the many-valued semantics of the logics discussed above, 
 defined by means of twist structures based on non-deterministic  two-valued semantics. 
 In Section~\ref{sec.examples}, we will illustrate, using the set $\LAT $ above and the corresponding matrix logics, how many-logic model structures can be used to represent connections between databases, users with different types of access (expressed by different logics) to a common database, and the evolution of databases over time. The reader may skip Section~\ref{sec.twist.and.semantics} and proceed directly to Section~\ref{sec.examples}, if desired. 
  However, it is important to make some observations concerning the lattices in $\LAT $:


  \begin{remark} \ 
  	
  	\m\mh (i)  Each $L_w\in \LAT $ is a  dc-sublattice  of $L_6$ then, by Proposition \ref{Prop3}, 
  $$v_w(\square A) \geq \big(\bigwedge_{L_6} \{ v_{w'}(A) \ : \ w R w'\} \big)^{L_w}.$$
  (ii) $L_6$ is finite and distributive, then it is a complete and  a $[\land\bigvee]$-lattice, and so are the  dc-sublattices in $\LAT$. So, by Proposition \ref{Prop4},  $x^{L_w} \cdot_{L_w} y^{L_w} = (x \cdot_{L_6} y)^{L_w}$. Therefore, being $L_6$ finite, it holds that
  $$v_w(\square A) = \big(\bigwedge_{L_6} \{ v_{w'}(A) \ : \ w R w'\} \big)^{L_w}.$$

\end{remark}

 

 \section{Semantics} \label{sec.twist.and.semantics}

We start by  providing two-valued non-deterministic semantics for the logics above.   
Then, the respective many-valued semantics 
 will be defined by means of twist structures, based on the two-valued semantics. 
Note that the two-valued semantics are non-deterministic in the sense that the semantic value of a complex formula is not functionally determined by the values of its parts -- for example, when $\rho(p)=1$, $\rho(\neg p)$ may be either $0$ or $1$ (cf. Definition~\ref{def.sem.two.val}). The many-valued semantics, on the other hand, are of course deterministic 
(cf.~Section~\ref{sec.twist}).


  Consider the set  
\begin{center}
 $\mathbb{L} = \{  \letkp, \fdetobot, \lj, \lpto, \jt, \ktto, \lt, CL^s, CL^w\}$. 
    \end{center}
 All the logics in $\mathbb{L}$ will be defined below over the  signature $\Sigma = \{\con, \neg,\to,\land,\lor  \}$, 
 so they have the same language of \letkp.   


 
\subsection{Two-valued semantics}

\begin{definition} [Two-valued semantics for $\ast \in \mathbb{L}$]  \label{def.sem.two.val} 

\mmm\mh Consider the following semantic clauses: 
  
\begin{itemize} 
    \item[(v1)] $\rho(A \land  B )=1$ iff $\rho(A)=1$ and $\rho( B )=1$,
\item[(v2)] $\rho(A \lor  B )=1$ iff $\rho(A)=1$ or $\rho( B )=1$,
\item[(v3)]  $\rho(A \to  B ) = 1$ iff $\rho(A) = 0$ or  $\rho(B ) = 1$,
\item[(v4)] $\rho(\neg (A \land  B ))$ = 1 iff $\rho(\neg A) = 1$ or  $\rho(\neg  B ) = 1$,
\item[(v5)] $\rho(\neg (A \lor  B )) = 1$ iff $\rho(\neg A) = 1$ and  $\rho(\neg  B ) = 1$,
\item[(v6)]  $\rho(\neg(A \to  B )) = 1$ iff $\rho(A) = 1$ and  $\rho(\neg B) = 1$,
\item[(v7)] $\rho(A) = 1$ iff $\rho(\neg\neg A) = 1$, 
\item[(v8)]  $\rho(\con A) = 0$,   
\item[(v9)]  $\rho(\neg\con A) = 1$, 
\item[(v10)]  $\rho(\con A) = 1$ iff $\big( \rho(A) = 1$ iff  $\rho(\neg A) = 0 \big)$,
\item[(v11)]  $\rho(\con A) = 1$ iff $\rho (\neg\con A) = 0$, 
\item[(v12)]   If $\rho(\neg A)=0$, then $\rho(A)=1$, 
\item[(v13)]   If $\rho(\neg A)=1$, then $\rho(A)=0$, 
\item[(v14)] $\rho(\neg A) = 1$ iff $\rho( A) = 1$, 
\item[(v15)]  $\rho(\con A) = 1$,   
\item[(v16)]  If $\rho(\con A) = 1$, then  $ \rho(A) = 1$ iff  $\rho(\neg A) = 0 $, 
\item[(v17)] $\rho(\con\con A)=1$,
\item[(v18)]  $\rho(\con A) = 1  $ iff $ \rho (\con\neg A) = 1$, 

\item[(v19)] $\rho (\con (A \land  B ))$ = 1 iff 

\quad  $\rho(\con A) = \rho(\con B) = \rho(A) = \rho(B) =1$ or  

\quad  $\rho( \con A )=\rho(\neg A) = 1$ or 
 
\quad  $\rho( \con B)=\rho(\neg B) = 1$,

\item[(v20)] $\rho ( \con(A \lor  B ))$ = 1 iff 

\quad $\rho(\con A) = \rho(\con B) = \rho(\neg A) = \rho(\neg B) =1$ or  

\quad  $\rho(\con A )=\rho(A) = 1$ or 

\quad  $\rho( \con B)=\rho( B) = 1$,

\item[(v21)] $\rho (\con (A \to   B ))$ = 1 iff

\quad $\rho(A) = v(\con B) =  \rho(\neg B) =1$ or  
 
 \quad  $\rho( \con A )=\rho(\neg A) = 1$ or 
 
 \quad  $\rho( \con B)=v(B) = 1$.

\item[(v22)] $\rho (\con (A \to   B ))$ = 1 iff

\quad $\rho(A) =0$ or  $\rho(\cons B)=1$.

\end{itemize}

\m

\mh  A  bivaluation for a logic $\ast\in \mathbb{L}$  is a  
function $\rho:  \mathcal{L} \to \{0,1\}$  satisfying the clauses:

\begin{itemize}

     \item[-] v1 to v7 and v16 to v21, for $\ast = \letkp$, 
    \item[-] v1  to v9,  for  $\ast = FDE^\to_\bot$,    

    \item[-] v1 to v7, v10, v11, and v22, for $\ast = \lj$, 

    \item[-] v1  to v9 and v12, for  $\ast = LP^\to_\bot$,
 
    \item[-] v1 to v7, v10 to v12, and v22, for $\ast = \jt$, 
 
    \item[-] v1 to v9 and v13, for $\ast= \ktto$, 
 
    \item[-] v1 to v7,  v10, v11, v13, and v22, for $\ast=\lt$,  
 
    \item[-]  v1, v3, v4, v8, v14, for $\ast = CL^w$,  
 
    \item[-]  v1, v3, v4, v14, v15, for $\ast = CL^s$. 
\end{itemize}

\mh For all the logics $ \ast\in  \mathbb{L}$
 defined above, 
the semantical consequence relation  with respect to  bivaluations,  
$\models_\ast^2$, is defined as follows:  
$\Gamma\models_{\ast}^2 A$ \ if and only if for every bivaluation $\rho$ for  
$ \ast $, if $\rho(B)=1$ for every $B \in \Gamma$, then $\rho(A)=1$.

\end{definition}

\begin{remark} \ 

\m\mh (i)    The difference between \jt\ and \lt\ is that clause (v12), which makes excluded middle valid, holds in 
the former but not in the latter, while (v13), which makes explosion valid, holds in the latter 
but not in the former. 
Neither (v12) nor (v13) holds in \lj. 

\m\mh (ii) Both \jt\ and \lt\ can be defined in a language with 
\con,  $\lor$, and $\neg$. It would be enough to adopt only the clauses (v2) and (v5) instead of (v1) to (v6) 
and  to define $\land$ and $\to$, as is done in \cite{dottaviano.1970}.  

\m\mh (iii) A  bottom particle $\bot$ can be defined in \letkp, \lj, and \jt\  as $A\land\neg A \land\con A$, for any 
 formula $A$. In $FDE^\to_\bot$  and  $LP^\to_\bot$ it is defined as $\con A$.  
 \ktto\  and \lt\ already have  bottom particles, defined as $A\land\neg A$.

 \m\mh (iv) Deductive systems for all the logics  can be easily defined just by adding to the deductive systems 
 of \fde, \textit{K3}, and \textit{LP} appropriate rules for $\to$ 
 and $\bot$.  \lj, under the name \textit{BS4}, is axiomatized in \cite{omori.waragai}.

\end{remark}

\subsection{Twist structures} \label{sec.twist}

  
As we have seen, the logic \letkp\ admits a (non-deterministic) two-valued semantics. 
A twist structure for  \letkp\ is an  algebra whose domain is formed by
triples   $(z_1,z_2,z_3)$, called snapshots, over 
the two-element Boolean algebra with domain {\bf 2}.   
Each snapshot represents a three-dimensional semantic value 
 where     $z_1$, $z_2$, and $z_2$ are the semantic values of formulas $A$, $\neg A$, and $\con A$ in a given   bivaluation $\rho$,  that is,   $(\rho(A),\rho(\neg A),\rho(\con A))$.  
The semantic values $T, T_0, \nei, \bo, F_0, F$ correspond to the triples 
$(\rho(A),\rho(\neg A),\rho(\con A))$ that represent the six scenarios of \letkp.   
 All the logics related to dc-sublattices of $ L6 $ are {matrix logics} that will be  defined by adapting the semantics of \letkp\ to the four, three, and two scenarios being represented by the corresponding logic.

\begin{definition}[Twist structures  semantics for $\ast \in \mathbb{L}$]    \label{def.twist.geral} 

\m\mh Let   ${\bf 2}^3$ be the set of triples (snapshots) $z=(z_1,z_2,z_3)$ over ${\bf 2}$. 
The twist structure    $\mathcal{M}_\ast = \langle \textsf{B}_\ast,  {\rm  D}_\ast, \mathcal{O}_\ast \rangle$  for $\ast \in {\mathbb{L}}$ 
 {(over the Boolean algebra $\mathcal{B}_2$)}  
  is defined as follows:  
   
 \begin{itemize} 
\item[i.]  The set   $\textsf{B}_\ast$    
is the domain of $\mathcal{M}_\ast$, the set of semantic values: 

\item[] \ For $\ast=\letkp  =  \{T, \, T_0, \, \bo, \, \nei, \, F_0, \, F  \}$, 

\item[] \ For $\ast={\fdetobot}  =  \{ T_0, \, \bo, \, \nei, \, F_0  \}$, 

\item[] \ For $\ast={\lj}  =  \{ T, \, \bo, \, \nei, \, F  \}$, 

\item[] \ For $\ast={\ktto}  =  \{ T_0, \,  \nei, \, F_0  \}$, 

\item[] \ For $\ast={\lt}  =  \{ T, \,  \nei, \, F  \}$, 

\item[] \ For $\ast={\lpto}  =  \{ T_0, \,  \bo, \, F_0  \}$, 

\item[] \ For $\ast={\jt}  =  \{ T, \,  \bo, \, F  \}$, 

\item[] \ For $\ast={CL^w}  =  \{ T_0, \,   F_0  \}$, 
\item[] \ For $\ast={CL^s}  =  \{ T, \,   F  \}$,

\m\item[]  where  \ $T=(1,0,1),  T_{0}=(1,0,0),  \bo=(1,1,0), \nei=(0,0,0), F_{0}=(0,1,0), $ and $ F=(0,1,1) $.

\mm\item[ii.] The set ${\rm D}_\ast \neq \emptyset$, $ {\rm D}_\ast  \subseteq \textsf{B}_\ast$, is the set of designated  values: 
 $ {\rm D}_\ast = \{z \in \textsf{B}_\ast \ : \ z_1=1 \}$;   
  the set of non-designated semantic values is 
$ {\rm  ND}_\ast = \{z \in \textsf{B}_\ast  \ : \ z_1\neq 1 \} $.

\mm\item[iii.] $\mathcal{O}_\ast$ is a map that  assigns, to each $n$-ary connective $\#$ of $\Lg$, a function 
$\tilde{\#} :\textsf{B}_{\ast}^n \to \textsf{B}_\ast$, 
 defined as follows, for every $z$ and $w$ in $\textsf{B}_\ast$: 

 \begin{itemize}   

\m\item[] For $\ast \in  \lbb$: 
    \item[(1)]    $(z_1,z_2,z_3)\,\tilde{\land}\,(w_1,w_2,w_3) = (z_1\sqcap w_1,z_2\sqcup w_2,(z_1 \sqcap z_3 \sqcap w_1 \sqcap w_3) \sqcup (z_2 \sqcap z_3) \sqcup (w_2 \sqcap w_3))$;

    

	\item[(2)] $(z_1,z_2,z_3)\,\tilde{\lor}\,(w_1,w_2,w_3) = (z_1\sqcup w_1,z_2\sqcap w_2, (z_2 \sqcap z_3 \sqcap w_2 \sqcap w_3) \sqcup (z_1 \sqcap z_3) \sqcup (w_1 \sqcap w_3))$;
    

\item[(3)] $\tilde{\neg}\,(z_1,z_2,z_3) = (z_2,z_1,z_3 )$;

\m\item[] For $\ast \in \{\letkp,  \fdetobot,  \lpto, \ktto, CL^w, CL^s \}$: 
\item[(4)]  
 $(z_1,z_2,z_3)\,\tilde{\to}\,(w_1,w_2,w_3) = (\sneg z_1 \sqcup w_1,z_1\sqcap w_2, (z_1 \sqcap w_2 \sqcap w_3) \sqcup (z_2 \sqcap z_3) \sqcup (w_1 \sqcap w_3))$;


\m\item[] For $\ast \in \{\lj,  \jt,  \lt \}$: 
\item[(5)]  $(z_1,z_2,z_3)\,\tilde{\to}\,(w_1,w_2,w_3) = (\sneg z_1 \sqcup w_1,z_1\sqcap w_2, {\sim} z_1 \sqcup w_3$);





\m\item[] For $\ast \in \{\letkp,  \lj, \jt, \lt, CL^s \}$: 

 \item[(6)] $\tilde{\circ}\,(z_1,z_2,z_3) = (z_3,\sneg z_3 ,1 )$;

\m \item[] For $\ast \in \{\fdetobot, \ktto, \lpto, CL^w  \}$: 

\m \item[(7)] $\tilde{\circ}\,(z_1,z_2,z_3) = (z_3,\sneg z_3 , 0 )$.

\end{itemize}
\end{itemize}

\mh For all the logics $ \ast\in  \mathbb{L}$ defined above, 
 the semantic consequence relation with respect to  $\mathcal{M}_\ast$,
  denoted by $\vDash_{\ast} $, is defined as follows:
$ \Gamma\models_{\ast}  A$  if and only if, for every valuation $v$ over $\mathcal{M}_\ast$, 
if $v(B) \in  \rm  D_\ast$ for every $B \in \Gamma$, then $v(A) \in {\rm D}_\ast$.   
\end{definition}




\begin{definition}[Many-valued semantics for  $ \ast \in \mathbb{L} $] \label{def.valuation}

 \m\mh A \textit{valuation} $v$ over the twist structure $ \mathcal{M}_\ast $ is a function  $v:\mathcal{L} \to \mathsf{B}_\ast $ such that:   
\begin{itemize}  
	\item[(v1)] $  v(A \land B) \ = \ v(A) \ \tilde{\land} \ v(B)$;
	\item[(v2)] $  v(A \lor B) \ = \ v(A) \ \tilde{\lor} \ v(B)$;
	\item[(v3)] $  v(A \to B) \ = \  v(A) \ \tilde{\to} \ v(B)$;
	\item[(v4)] $  v(\neg A) \ = \ \tilde{\neg} v(A)$;
	\item[(v5)] $  v(\con A) \ = \ \tilde{\con} v(A)$.

\end{itemize}

\end{definition}

\begin{remark} \

  \m\mh (i) Given the definition of $\textsf{B}_\ast $,  we can write $v(A) = (v_1, v_2, v_3)$, where
 $v_i \in \{ 0,1 \}$.  
 
 \m\mh (ii) The domains $ \textsf{B}_\ast $ do not contain the triples (0, 0, 1) and (1, 1, 1) precisely because paracomplete and paraconsistent scenarios  do not have reliable information, which is expressed by 1 in the third coordinate.

\m\mh (iii)  We saw that each snapshot  is a triple  $(z_1,z_2,z_3)$ such that  
$z_1 = \rho(A)$, $z_2 = \rho(\neg A) $,
and $z_3 = \rho(\con A) $ in a two-valued semantics $ \rho $. Thus, 
  the semantic values $T, T_0, \bo,\nei, F_0 $, and $F$
can be thought of as \textit{names} of the six scenarios (i) to (vi) expressed by \lets.
 
\m\mh (iv) There is no need of twist 
structures to define the semantics of classical logic. 
 Since the snapshots always  have 1 in the third position in the case of $CL^s$ and 0 in $CL^w$, 
 the operations of the connectives always yields in the third position 1 and 0, in each case. 
 Regarding a formula $\con A$,  it is always assigned $(1,0,1) = T$ in $CL^s$, and  
    $(0,1,0) = F_0$ in $CL^w$.  
 Therefore, in $CL^s$, $\con A$ is a top particle, and in  $CL^w$,  it is a bottom particle.


\end{remark}

\subsection{The  matrices}
 
 The logics discussed above can be defined using matrices.

\begin{definition}   \  \label{def.matrices}
 
\m\mh For $  \ast \in \mathbb{L} $,   
the semantics defined by $ \mathcal{M}_\ast $ can be expressed by means of the following matrices: 


\begin{itemize}

 \item[1.] Four-valued semantics for \fdetobot\ ($D=\{T_0,\bo\}$)

\begin{center} \footnotesize
\begin{tabular}{|c||c|c|c|c|}
\hline
 $\tilde{\wedge}$ & $T_0$  & $\bo$   & $\nei$ & $F_0$\\[1mm]
 \hline \hline
     $T_0$    & $T_0$  & $\bo $ & $\nei$ & $F_0$  \\[1mm] \hline
     $\bo$    & $\bo$  & $\bo$ & $F_0$ & $F_0$    \\[1mm] \hline
     $\nei$    & $\nei$  & $F_0$ & $\nei$ & $F_0$    \\[1mm] \hline
     $F_0$    & $F_0$  & $F_0$ & $F_0$ & $F_0$  \\[1mm] \hline
\end{tabular}
\hspace{1cm}
\begin{tabular}{|c||c|c|c|c|}
\hline
 $\tilde{\lor}$ & $T_0$  & $\bo$   & $\nei$ & $F_0$\\[1mm]
 \hline \hline
     $T_0$    & $T_0$  & $T_0$ & $T_0$ & $T_0$  \\[1mm] \hline
     $\bo$    & $T_0$ & $\bo$ & $T_0$ & $\bo$    \\[1mm] \hline
     $\nei$    & $T_0$ & $T_0$ & $\nei$ & $\nei$    \\[1mm] \hline
     $F_0$    & $T_0$  & $\bo$ & $\nei$ & $F_0$  \\[1mm] \hline
\end{tabular}

\

\

\begin{tabular}{|c||c|c|c|c|}
\hline
 $\tilde{\to}$ &   $T_0$  & $\bo$ & $\nei$ & $F_0$    \\[1mm]
 \hline \hline
     $T_0$   & $T_0$ & $\bo$ & $\nei$ & $F_0$   \\[1mm] \hline
     $\bo$    & $T_0$ & $\bo$ & $\nei$ & $F_0$   \\[1mm] \hline
     $\nei$    & $T_0$ & $T_0$ & $T_0$ & $T_0$   \\[1mm] \hline
     $F_0$     & $T_0$ & $T_0$ & $T_0$ & $T_0$   \\[1mm] \hline
\end{tabular}
\hspace{1.5cm}
\begin{tabular}{|c||c|} \hline
$\quad$ & $\tilde{\neg}$ \\[1mm]
 \hline \hline
     $T_0$   & $F_0$    \\[1mm] \hline
     $\bo$   &$\bo$    \\[1mm] \hline
     $\nei$   & $\nei$    \\[1mm] \hline
     $F_0$   & $T_0$    \\[1mm] \hline
\end{tabular}
\hspace{1.5cm}
\begin{tabular}{|c||c|}
\hline
 $\quad$ & $\tilde{\circ}$ \\[1mm]
 \hline \hline
     $T_0$   & $F_0$     \\[1mm] \hline
     $\bo$   & $F_0$     \\[1mm] \hline
     $\nei$   & $F_0$     \\[1mm] \hline
     $F_0$   & $F_0$     \\[1mm] \hline
\end{tabular}
\end{center}

\item[2.] 
Three-valued semantics for \ktto\ ($D=\{T_0\}$)

\begin{center} \footnotesize
\begin{tabular}{|c||c|c|c|}
\hline
 $\tilde{\land}$ & $T_0$     & $\nei$ & $F_0$\\[1mm]
 \hline \hline
     $T_0$    & $T_0$   & $\nei$ & $F_0$  \\[1mm] \hline
     $\nei$    & $\nei$  & $\nei$ & $F_0$    \\[1mm] \hline
     $F_0$    & $F_0$  &  $F_0$ & $F_0$  \\[1mm] \hline
\end{tabular}
\hspace{1cm}
\begin{tabular}{|c||c|c|c|}
\hline
 $\tilde{\lor}$ & $T_0$     & $\nei$ & $F_0$\\[1mm]
 \hline \hline
     $T_0$    & $T_0$   & $T_0$ & $T_0$  \\[1mm] \hline
     $\nei$    & $T_0$  & $\nei$ & $\nei$    \\[1mm] \hline
     $F_0$    & $T_0$  &  $\nei$ & $F_0$  \\[1mm] \hline
\end{tabular}

\

\

\begin{tabular}{|c||c|c|c|}
\hline
 $\tilde{\to}$ &   $T_0$  &  $\nei$ & $F_0$    \\[1mm]
 \hline \hline
     $T_0$   & $T_0$  & $\nei$ & $F_0$   \\[1mm] \hline
     $\nei$    & $T_0$ & $T_0$ & $T_0$   \\[1mm] \hline
     $F_0$     & $T_0$  & $T_0$ & $T_0$   \\[1mm] \hline
\end{tabular}
\hspace{1.5cm}
\begin{tabular}{|c||c|} \hline
$\quad$ & $\tilde{\neg}$ \\[1mm]
 \hline \hline
     $T_0$   & $F_0$    \\[1mm] \hline
     $\nei$   & $\nei$    \\[1mm] \hline
     $F_0$   & $T_0$    \\[1mm] \hline
\end{tabular}
\hspace{1.5cm}
\begin{tabular}{|c||c|}
\hline
 $\quad$ & $\tilde{\circ}$ \\[1mm]
 \hline \hline
     $T_0$   & $F_0$     \\[1mm] \hline
     $\nei$   & $F_0$     \\[1mm] \hline
     $F_0$   & $F_0$     \\[1mm] \hline
\end{tabular}
\end{center}

\mh The tables above (except the one for \con) can be found in  \cite{hazen.2019}, where the logic \ktto\ is investigated. 
 Note that $\con$ can be defined in the original language of \ktto\ as follows:    $\con A\defi A\land\neg A$.

\item[3.] Three-valued semantics for \lpto\ ($D=\{T_0,\bo\}$)

 \begin{center} \footnotesize 
\begin{tabular}{|c||c|c|c|}
\hline
 $\tilde{\land}$ & $T_0$     & $\bo$ & $F_0$\\[1mm]
 \hline \hline
     $T_0$    & $T_0$   & $\bo$ & $F_0$  \\[1mm] \hline
     $\bo$    & $\bo$  & $\bo$ & $F_0$    \\[1mm] \hline
     $F_0$    & $F_0$  &  $F_0$ & $F_0$  \\[1mm] \hline
\end{tabular}
\hspace{1cm}
\begin{tabular}{|c||c|c|c|}
\hline
 $\tilde{\lor}$ & $T_0$     & $\bo$ & $F_0$\\[1mm]
 \hline \hline
     $T_0$    & $T_0$   & $T_0$ & $T_0$  \\[1mm] \hline
     $\bo$    & $T_0$  & $\bo$ & $\bo$    \\[1mm] \hline
     $F_0$    & $T_0$  &  $\bo$ & $F_0$  \\[1mm] \hline
\end{tabular}

\

\

\begin{tabular}{|c||c|c|c|}
\hline
 $\tilde{\to}$ &   $T_0$  &  $\bo$ & $F_0$    \\[1mm]
 \hline \hline
     $T_0$   & $T_0$  & $\bo$ & $F_0$   \\[1mm] \hline
     $\bo$    & $T_0$ & $\bo$ & $F_0$   \\[1mm] \hline
     $F_0$     & $T_0$  & $T_0$ & $T_0$   \\[1mm] \hline
\end{tabular}
\hspace{1.5cm}
\begin{tabular}{|c||c|} \hline
$\quad$ & $\tilde{\neg}$ \\[1mm]
 \hline \hline
     $T_0$   & $F_0$    \\[1mm] \hline
     $\bo$   & $\bo$    \\[1mm] \hline
     $F_0$   & $T_0$    \\[1mm] \hline
\end{tabular}
\hspace{1.5cm}
\begin{tabular}{|c||c|}
\hline
 $\quad$ & $\tilde{\circ}$ \\[1mm]
 \hline \hline
     $T_0$   & $F_0$     \\[1mm] \hline
     $\bo$   & $F_0$     \\[1mm] \hline
     $F_0$   & $F_0$     \\[1mm] \hline
\end{tabular}
\end{center}

\mh The logic $LP^\to$, defined by the tables above for $\neg,\lor,\land$, and  $\to$, is 
investigated in \citep{hazen.2019}.  
 The logic \lpto\ extends $LP^\to$ with a bottom particle, defined here as $\con A$ for any $A$. 
 We are not aware of a previous investigation of the logic \lpto.

\item[4.] Four-valued semantics for \lj\  ($D=\{T,\bo\}$)

 \begin{center}
 \footnotesize
\begin{tabular}{|c||c|c|c|c|}
\hline
 $\tilde{\wedge}$ & $T$  & $\bo$   & $\nei$ & $F$\\[1mm]
 \hline \hline
     $T$    & $T$  & $\bo $ & $\nei$ & $F$  \\[1mm] \hline
     $\bo$    & $\bo$  & $\bo$ & $F$ & $F$    \\[1mm] \hline
     $\nei$    & $\nei$  & $F$ & $\nei$ & $F$    \\[1mm] \hline
     $F$    & $F$  & $F$ & $F$ & $F$  \\[1mm] \hline
\end{tabular}
\hspace{1cm}
\begin{tabular}{|c||c|c|c|c|}
\hline
 $\tilde{\lor}$ & $T$  & $\bo$   & $\nei$ & $F$\\[1mm]
 \hline \hline
     $T$    & $T$  & $T$ & $T$ & $T$  \\[1mm] \hline
     $\bo$    & $T$ & $\bo$ & $T$ & $\bo$    \\[1mm] \hline
     $\nei$    & $T$ & $T$ & $\nei$ & $\nei$    \\[1mm] \hline
     $F$    & $T$  & $\bo$ & $\nei$ & $F$  \\[1mm] \hline
\end{tabular}

\

\

\begin{tabular}{|c||c|c|c|c|}
\hline
 $\tilde{\to}$ &   $T$  & $\bo$ & $\nei$ & $F$    \\[1mm]
 \hline \hline
     $T$   & $T$ & $\bo$ & $\nei$ & $F$   \\[1mm] \hline
     $\bo$    & $T$ & $\bo$ & $\nei$ & $F$   \\[1mm] \hline
     $\nei$    & $T$ & $T$ & $T$ & $T$   \\[1mm] \hline
     $F$     & $T$ & $T$ & $T$ & $T$   \\[1mm] \hline
\end{tabular}
\hspace{1.5cm}
\begin{tabular}{|c||c|} \hline
$\quad$ & $\tilde{\neg}$ \\[1mm]
 \hline \hline
     $T$   & $F$    \\[1mm] \hline
     $\bo$   &$\bo$    \\[1mm] \hline
     $\nei$   & $\nei$    \\[1mm] \hline
     $F$   & $T$    \\[1mm] \hline
\end{tabular}
\hspace{1.5cm}
\begin{tabular}{|c||c|}
\hline
 $\quad$ & $\tilde{\circ}$ \\[1mm]
 \hline \hline
     $T$   & $T$     \\[1mm] \hline
     $\bo$   & $F$     \\[1mm] \hline
     $\nei$   & $F$     \\[1mm] \hline
     $F$   & $T$     \\[1mm] \hline
\end{tabular}
\end{center}

\item[5.] Three-valued semantics for \jt/\textit{LFI1} ($D=\{T,\bo\}$)

  \begin{center} \footnotesize
\begin{tabular}{|c||c|c|c|}
\hline
 $\tilde{\land}$ & $T$     & $\bo$ & $F$\\[1mm]
 \hline \hline
     $T$    & $T$   & $\bo$ & $F$  \\[1mm] \hline
     $\bo$    & $\bo$  & $\bo$ & $F$    \\[1mm] \hline
     $F$    & $F$  &  $F$ & $F$  \\[1mm] \hline
\end{tabular}
\hspace{1cm}
\begin{tabular}{|c||c|c|c|}
\hline
 $\tilde{\lor}$ & $T$     & $\bo$ & $F$\\[1mm]
 \hline \hline
     $T$    & $T$   & $T$ & $T$  \\[1mm] \hline
     $\bo$    & $T$  & $\bo$ & $\bo$    \\[1mm] \hline
     $F$    & $T$  &  $\bo$ & $F$  \\[1mm] \hline
\end{tabular}
\hspace{1cm}

\

\

\begin{tabular}{|c||c|c|c|}
\hline
 $\tilde{ \to}$ &   $T$  &  $\bo$ & $F$    \\[1mm]
 \hline \hline
     $T$   & $T$  & $\bo$ & $F$   \\[1mm] \hline
     $\bo$    & $T$ & $\bo$ & $F$   \\[1mm] \hline
     $F$     & $T$  & $T$ & $T$   \\[1mm] \hline
\end{tabular}
\hspace{1.5cm}
\begin{tabular}{|c||c|} \hline
$\quad$ & $\tilde{\neg}$ \\[1mm]
 \hline \hline
     $T$   & $F$    \\[1mm] \hline
     $\bo$   & $\bo$    \\[1mm] \hline
     $F$   & $T$    \\[1mm] \hline
\end{tabular}
\hspace{1.5cm}
\begin{tabular}{|c||c|}
\hline
 $\quad$ & $\tilde{\circ}$ \\[1mm]
 \hline \hline
     $T$   & $T$     \\[1mm] \hline
     $\bo$   & $F$     \\[1mm] \hline
     $F$   & $T$     \\[1mm] \hline
\end{tabular}

\end{center}

\item[6.] Three-valued semantics for \lt\ ($D=\{T\}$)

\begin{center} \footnotesize 
\begin{tabular}{|c||c|c|c|}
\hline
 $\tilde{\land}$ & $T$     & $\nei$ & $F$\\[1mm]
 \hline \hline
     $T$    & $T$   & $\nei$ & $F$  \\[1mm] \hline
     $\nei$    & $\nei$  & $\nei$ & $F$    \\[1mm] \hline
     $F$    & $F$  &  $F$ & $F$  \\[1mm] \hline
\end{tabular}
\hspace{1cm}
\begin{tabular}{|c||c|c|c|}
\hline
 $\tilde{\lor}$ & $T$     & $\nei$ & $F$\\[1mm]
 \hline \hline
     $T$    & $T$   & $T$ & $T$  \\[1mm] \hline
     $\nei$    & $T$  & $\nei$ & $\nei$    \\[1mm] \hline
     $F$    & $T$  &  $\nei$ & $F$  \\[1mm] \hline
\end{tabular}
\hspace{1.5cm}

\

\

\begin{tabular}{|c||c|c|c|}
\hline
 $\tilde{ \to}$ &   $T$  &  $\nei$ & $F$    \\[1mm]
 \hline \hline
     $T$   & $T$  & $\nei$ & $F$   \\[1mm] \hline
     $\nei$    & $T$ & $T$ & $T$   \\[1mm] \hline
     $F$     & $T$  & $T$ & $T$   \\[1mm] \hline
\end{tabular}
\hspace{1.5cm}
\begin{tabular}{|c||c|} \hline
$\quad$ & $\tilde{\neg}$ \\[1mm]
 \hline \hline
     $T$   & $F$    \\[1mm] \hline
     $\nei$   & $\nei$    \\[1mm] \hline
     $F$   & $T$    \\[1mm] \hline
\end{tabular}
\hspace{1.5cm}
\begin{tabular}{|c||c|}
\hline
 $\quad$ & $\tilde{\circ}$ \\[1mm]
 \hline \hline
     $T$   & $T$     \\[1mm] \hline
     $\nei$   & $F$     \\[1mm] \hline
     $F$   & $T$     \\[1mm] \hline
\end{tabular}

\end{center}

\begin{remark} Regarding the presentation of the logics \jt\ and \lt\ above, the following remarks
 are noteworthy.

\m\mh (i) The logic \jt\ was introduced by \cite{dottaviano.1970} in the language with $\neg $, $\lor$, and 
$\nabla$, with the following table for $\nabla A$:

\begin{center}
\begin{tabular}{|c||c|}
\hline
 $\quad$ & $\tilde{\nabla}$ \\[1mm]
 \hline \hline
     $T$   & $T$     \\[1mm] \hline
     $\bo $   & $T$     \\[1mm] \hline
     $F$   & $F$     \\[1mm] \hline
\end{tabular}
\end{center}
In  \cite{RE1990}, Chapter~IX,  it is shown
  that the operator \con\ can be defined  as $\neg(\nabla A \land \nabla\neg A) $.
  This    anticipates the presentation of \textit{J3} as the logic \textit{LFI1} in \cite{carn:marcos:deamo:2000}. 
  In the language with \con\ primitive, 
$\nabla A$ is defined as $A \lor \neg\con A $.

 \m\mh (ii) As remarked by \citet[p.~1350]{dottaviano.1970}, 
the logic \textit{\L3} can be defined in the language with $\neg, \lor$, and  $\nabla$ 
and the respective matrices of \jt, just   replacing \bo\ with \nei\ and taking $T$ as the only designated value. 
 Note that the implication $\to$ of our presentation of \lt\ \textit{is not} the original implication of \lt, call it $\to_L$,
 which is given by the table below: 
 
\begin{center}
    \begin{tabular}{|c||c|c|c|}
\hline
 $\tilde{ \to_L}$ &   $T$  &  $\nei$ & $F$    \\[1mm]
 \hline \hline
     $T$   & $T$  & $\nei$ & $F$   \\[1mm] \hline
     $\nei$    & $T$ & $T$ & $\nei $   \\[1mm] \hline
     $F$     & $T$  & $T$ & $T$   \\[1mm] \hline
\end{tabular}

\end{center} 
 
\mh  In \lt, $\nabla$ and \con\ are interdefinable as in \jt, while   $A\to B$  can be  defined as $\neg\nabla A\lor B$,  and the original implication of \lt\ $A{\to_L} B$   as $(\nabla\neg A \lor B) \land (\nabla B \lor\neg A) $, where 
$A\land B$ is $\neg(\neg A \lor\neg B)$. 
\lt\ can be also defined in a language with $\neg$ and $\to_L$ and the respective tables, being $T$ designated.  

  

%

\end{remark}

\item[7.] Six-valued semantics for \letkp\ ($D=\{T,T_0,\bo\}$)

\begin{center}
\footnotesize 
\begin{tabular}{|c||c|c|c|c|c|c|}
\hline
 $\tilde{\wedge}$ & $T$  & $T_0$   & $\bo$ & $\nei$ & $F_0$  & $F$\\[1mm]
 \hline \hline
    $T$    & $T$  & $T_0$ & $\bo$ & $\nei$ & $F_0$ & $F$   \\[1mm] \hline
     $T_0$    & $T_0$  & $T_0$ & $\bo$ & $\nei$ & $F_0$ & $F$  \\[1mm] \hline
     $\bo$    & $\bo$  & $\bo$ & $\bo$ & $F_0$ & $F_0$ & $F$  \\[1mm] \hline
     $\nei$    & $\nei$  & $\nei$ & $F_0$ & $\nei$ & $F_0$ & $F$  \\[1mm] \hline
     $F_0$    & $F_0$  & $F_0$ & $F_0$ & $F_0$ & $F_0$ & $F$  \\[1mm] \hline
     $F$    & $F$  & $F$ & $F$ & $F$ & $F$ & $F$  \\[1mm] \hline
\end{tabular}
\hspace{0.5cm}
\begin{tabular}{|c||c|c|c|c|c|c|}
\hline
 $\tilde{\vee}$ & $T$  & $T_0$  & $\bo$ & $\nei$ & $F_0$  & $F$ \\[1mm]
 \hline \hline
    $T$    & $T$  & $T$ & $T$ & $T$ & $T$ & $T$   \\[1mm] \hline
     $T_0$    & $T$  & $T_0$ & $T_0$ & $T_0$ & $T_0$ & $T_0$   \\[1mm] \hline
     $\bo$    & $T$  & $T_0$ & $\bo$ & $T_0$ & $\bo$ & $\bo$   \\[1mm] \hline
     $\nei$    & $T$  & $T_0$ & $T_0$ & $\nei$ & $\nei$ & $\nei$   \\[1mm] \hline
     $F_0$    & $T$  & $T_0$ & $\bo$ & $\nei$ & $F_0$ & $F_0$   \\[1mm] \hline
     $F$    & $T$  & $T_0$ & $\bo$ & $\nei$ & $F_0$ & $F$  \\[1mm] \hline
\end{tabular}

\

\

\begin{tabular}{|c||c|c|c|c|c|c|}
\hline
 $\tilde{\to}$ & $T$  & $T_0$  & $\bo$ & $\nei$ & $F_0$  & $F$ \\[1mm]
 \hline \hline
    $T$    & $T$  & $T_0$ & $\bo$ & $\nei$ & $F_0$ & $F$   \\[1mm] \hline
     $T_0$    & $T$  & $T_0$ & $\bo$ & $\nei$ & $F_0$ & $F$  \\[1mm] \hline
     $\bo$    & $T$  & $T_0$ & $\bo$ & $\nei$ & $F_0$ & $F$  \\[1mm] \hline
     $\nei$    & $T$  & $T_0$ & $T_0$ & $T_0$ & $T_0$ & $T_0$  \\[1mm] \hline
     $F_0$    & $T$  & $T_0$ & $T_0$ & $T_0$ & $T_0$ & $T_0$  \\[1mm] \hline
     $F$    & $T$  & $T$ & $T$ & $T$ & $T$ & $T$  \\[1mm] \hline
\end{tabular}
\hspace{1cm}
\begin{tabular}{|c||c|} \hline
$\quad$ & $\tilde{\neg}$ \\[1mm]
 \hline \hline
    $T$   & $F$    \\[1mm] \hline
     $T_0$   & $F_0$    \\[1mm] \hline
     $\bo$   &$\bo$    \\[1mm] \hline
     $\nei$   & $\nei$    \\[1mm] \hline
     $F_0$   & $T_0$    \\[1mm] \hline
     $F$   & $T$    \\[1mm] \hline
\end{tabular}
\hspace{1cm}
\begin{tabular}{|c||c|}
\hline
 $\quad$ & $\tilde{\circ}$ \\[1mm]
 \hline \hline
    $T$   & $T$    \\[1mm] \hline
     $T_0$   & $F$     \\[1mm] \hline
     $\bo$   & $F$     \\[1mm] \hline
     $\nei$   & $F$     \\[1mm] \hline
     $F_0$   & $F$     \\[1mm] \hline
     $F$   & $T$    \\[1mm] \hline
\end{tabular}
\end{center}

\end{itemize}

\end{definition}

\mh 
Note that the matrices of \jt\ and \lt\ are submatrices of the matrices of \lj, 
obtained from the latter by removing rows and columns with the values \nei\ and \bo, respectively.
Analogously, the matrices of \lpto\ and \ktto\ are submatrices of the matrices of \fdetobot. However, neither \lj\ nor 
\fdetobot\ are submatrices of \letkp.

 \section{On some examples}  \label{sec.examples}

 In this section, we will consider many-logic modal structures 
 based on  the lattice $L6$, defined by   the logic \letkp.  
 The  set $\LAT $  contains the dc-sublattices of  $L6$ with four, three, and two values we have seen in Section~\ref{sec.sublattices} above. 
 Recall that here we have just one logic associated with each lattice, and that we may  read $T_0$ and $F_0$ as  {just true} and {just false},   
 $T$ and $F$ as  {certified true} and {certified false}, 
   and that $v_w(\Box A)=x$  is read   as 
`$A$ has the semantic value $x$ for an agent accessing the information available from $w$'.

\begin{example} \ \label{ex.1}

\mm  $I(w_1) = L6, \ I(w_2) = \lw, \ I(w_3) = \bs$ 


\mm 
  $v_{w_1}(p)=T$ 
  
       $v_{w_2}(p)=T_0$   
       
 $v_{w_3}(p)=\bo$    

\mm        $(v_{w_2}(p))^{L_{w_1}}=T_0$   
       
$(v_{w_3}(p))^{L_{w_1}}=\bo  $

      \mm          
 $v_{w_1}(\Box p)= \bigwedge \{ T, T_0, \bo \} = \bo $

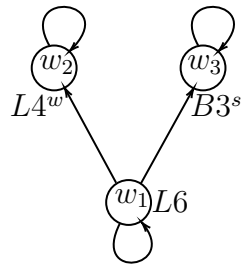
 \begin{figure}[H]
     \centering
    \begin{figure}[H]
		\centering
\tikzset{every picture/.style={line width=0.75pt}} 

\begin{tikzpicture}[x=0.75pt,y=0.75pt,yscale=-0.45,xscale=0.45]

\draw   (362,93) .. controls (362,79.19) and (373.19,68) .. (387,68) .. controls (400.81,68) and (412,79.19) .. (412,93) .. controls (412,106.81) and (400.81,118) .. (387,118) .. controls (373.19,118) and (362,106.81) .. (362,93) -- cycle ;
\draw   (196,93) .. controls (196,79.19) and (207.19,68) .. (221,68) .. controls (234.81,68) and (246,79.19) .. (246,93) .. controls (246,106.81) and (234.81,118) .. (221,118) .. controls (207.19,118) and (196,106.81) .. (196,93) -- cycle ;
\draw   (279,243) .. controls (279,229.19) and (290.19,218) .. (304,218) .. controls (317.81,218) and (329,229.19) .. (329,243) .. controls (329,256.81) and (317.81,268) .. (304,268) .. controls (290.19,268) and (279,256.81) .. (279,243) -- cycle ;
\draw    (317.59,222.03) -- (375.62,117.28) ;
\draw [shift={(376.59,115.53)}, rotate = 118.99] [color={rgb, 255:red, 0; green, 0; blue, 0 }  ][line width=0.75]    (10.93,-3.29) .. controls (6.95,-1.4) and (3.31,-0.3) .. (0,0) .. controls (3.31,0.3) and (6.95,1.4) .. (10.93,3.29)   ;
\draw    (290.09,221.03) -- (233.54,116.78) ;
\draw [shift={(232.59,115.03)}, rotate = 61.52] [color={rgb, 255:red, 0; green, 0; blue, 0 }  ][line width=0.75]    (10.93,-3.29) .. controls (6.95,-1.4) and (3.31,-0.3) .. (0,0) .. controls (3.31,0.3) and (6.95,1.4) .. (10.93,3.29)   ;
\draw    (381.59,67.93) .. controls (345.27,3.67) and (463.4,16.31) .. (404.98,73.57) ;
\draw [shift={(404.09,74.43)}, rotate = 316.21] [color={rgb, 255:red, 0; green, 0; blue, 0 }  ][line width=0.75]    (10.93,-3.29) .. controls (6.95,-1.4) and (3.31,-0.3) .. (0,0) .. controls (3.31,0.3) and (6.95,1.4) .. (10.93,3.29)   ;
\draw    (215.59,67.93) .. controls (179.27,3.67) and (297.4,16.31) .. (238.98,73.57) ;
\draw [shift={(238.09,74.43)}, rotate = 316.21] [color={rgb, 255:red, 0; green, 0; blue, 0 }  ][line width=0.75]    (10.93,-3.29) .. controls (6.95,-1.4) and (3.31,-0.3) .. (0,0) .. controls (3.31,0.3) and (6.95,1.4) .. (10.93,3.29)   ;

 \draw    (294.54,265.86) .. controls (260.32,332.45) and (361.92,316.58) .. (319.2,265.82) ;
 \draw [shift={(318.54,265.06)}, rotate = 49.16] [color={rgb, 255:red, 0; green, 0; blue, 0 }  ][line width=0.75]    (10.93,-3.29) .. controls (6.95,-1.4) and (3.31,-0.3) .. (0,0) .. controls (3.31,0.3) and (6.95,1.4) .. (10.93,3.29)   ;

\draw (197,134) node   [align=left] {\begin{minipage}[lt]{17.15pt}\setlength\topsep{0pt}
$\lw$
\end{minipage}};
\draw (355,241.47) node   [align=left] {\begin{minipage}[lt]{17.15pt}\setlength\topsep{0pt}
$L6$
\end{minipage}};
\draw (398,134) node   [align=left] {\begin{minipage}[lt]{17.15pt}\setlength\topsep{0pt}
$\bs$
\end{minipage}};
\draw (314,257) node   [align=left] {\begin{minipage}[lt]{17.15pt}\setlength\topsep{0pt}
$\displaystyle w_{1}$\\
\end{minipage}};
\draw (395,108) node   [align=left] {\begin{minipage}[lt]{17.15pt}\setlength\topsep{0pt}
$\displaystyle w_{3}$\\
\end{minipage}};
\draw (230,108) node   [align=left] {\begin{minipage}[lt]{17.15pt}\setlength\topsep{0pt}
$\displaystyle w_{2}$\\
\end{minipage}};

\end{tikzpicture}
	\end{figure} 
 \end{figure}

\begin{quote}

\mh 	In this example, the worlds $w_2$ and $w_3$ can be thought of as databases with only local access, but $w_1$ can access $w_2$ and $w_3$, as well as itself.  $w_1$ is a $L6$-world, so it can `see everything' in any other database. 
  $w_2$ and $w_3$ are an $\lw$-world and a $\bs$-world, respectively. 

 $w_1$ can be thought of as the main database, while  $w_3$ is database possibly  inconsistent but with all definite information, i.e. information only positive or only negative, considered as certified, and all the  information stored in  $w_2$ is considered as non-certified.

$p$ is assigned $T$ in $w_1$ and $T_0$ in $w_2$, but  since there is contradictory information 
 in $w_3$,  the value of  $\Box p$ for an agent accessing the all the information available in $w_1$ is $\bo$. 
\end{quote}

\end{example}

\begin{example} \  \label{ex.2}
 
\mm $I(w_1) = \cw, \ I(w_2) = L6, \ I(w_3) = \nw$ 


    \mm   $v_{w_1}(p)=T_0$, 
       
          $v_{w_2}(p)=T$,  
          
          $v_{w_3}(p)=\nei$    

 \mm       $(v_{w_2}(p))^{L_{w_1}}=T_0$, 
       
       $(v_{w_3}(p))^{L_{w_1}}=F_0$

     \mm     $v_{w_1}(\Box p)= \bigwedge \{ T_0, T_0, F_0 \} = F_0 $

 
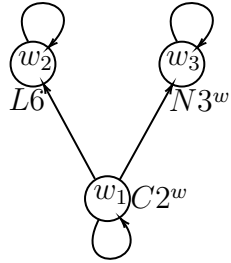
\begin{figure}[H]
     \centering
    \begin{figure}[H]
		\centering
\tikzset{every picture/.style={line width=0.75pt}} 

\begin{tikzpicture}[x=0.75pt,y=0.75pt,yscale=-0.45,xscale=0.45]

\draw   (362,93) .. controls (362,79.19) and (373.19,68) .. (387,68) .. controls (400.81,68) and (412,79.19) .. (412,93) .. controls (412,106.81) and (400.81,118) .. (387,118) .. controls (373.19,118) and (362,106.81) .. (362,93) -- cycle ;
\draw   (196,93) .. controls (196,79.19) and (207.19,68) .. (221,68) .. controls (234.81,68) and (246,79.19) .. (246,93) .. controls (246,106.81) and (234.81,118) .. (221,118) .. controls (207.19,118) and (196,106.81) .. (196,93) -- cycle ;
\draw   (279,243) .. controls (279,229.19) and (290.19,218) .. (304,218) .. controls (317.81,218) and (329,229.19) .. (329,243) .. controls (329,256.81) and (317.81,268) .. (304,268) .. controls (290.19,268) and (279,256.81) .. (279,243) -- cycle ;
\draw    (317.59,222.03) -- (375.62,117.28) ;
\draw [shift={(376.59,115.53)}, rotate = 118.99] [color={rgb, 255:red, 0; green, 0; blue, 0 }  ][line width=0.75]    (10.93,-3.29) .. controls (6.95,-1.4) and (3.31,-0.3) .. (0,0) .. controls (3.31,0.3) and (6.95,1.4) .. (10.93,3.29)   ;
\draw    (290.09,221.03) -- (233.54,116.78) ;
\draw [shift={(232.59,115.03)}, rotate = 61.52] [color={rgb, 255:red, 0; green, 0; blue, 0 }  ][line width=0.75]    (10.93,-3.29) .. controls (6.95,-1.4) and (3.31,-0.3) .. (0,0) .. controls (3.31,0.3) and (6.95,1.4) .. (10.93,3.29)   ;
\draw    (381.59,67.93) .. controls (345.27,3.67) and (463.4,16.31) .. (404.98,73.57) ;
\draw [shift={(404.09,74.43)}, rotate = 316.21] [color={rgb, 255:red, 0; green, 0; blue, 0 }  ][line width=0.75]    (10.93,-3.29) .. controls (6.95,-1.4) and (3.31,-0.3) .. (0,0) .. controls (3.31,0.3) and (6.95,1.4) .. (10.93,3.29)   ;
\draw    (215.59,67.93) .. controls (179.27,3.67) and (297.4,16.31) .. (238.98,73.57) ;
\draw [shift={(238.09,74.43)}, rotate = 316.21] [color={rgb, 255:red, 0; green, 0; blue, 0 }  ][line width=0.75]    (10.93,-3.29) .. controls (6.95,-1.4) and (3.31,-0.3) .. (0,0) .. controls (3.31,0.3) and (6.95,1.4) .. (10.93,3.29)   ;

 \draw    (294.54,265.86) .. controls (260.32,332.45) and (361.92,316.58) .. (319.2,265.82) ;
 \draw [shift={(318.54,265.06)}, rotate = 49.16] [color={rgb, 255:red, 0; green, 0; blue, 0 }  ][line width=0.75]    (10.93,-3.29) .. controls (6.95,-1.4) and (3.31,-0.3) .. (0,0) .. controls (3.31,0.3) and (6.95,1.4) .. (10.93,3.29)   ;

\draw (219,131.97) node   [align=left] {\begin{minipage}[lt]{17.15pt}\setlength\topsep{0pt}
$L6$
\end{minipage}};
\draw (355,241.47) node   [align=left] {\begin{minipage}[lt]{17.15pt}\setlength\topsep{0pt}
$\cw$
\end{minipage}};
\draw (398,133.47) node   [align=left] {\begin{minipage}[lt]{17.15pt}\setlength\topsep{0pt}
$\nw$
\end{minipage}};
\draw (314,257) node   [align=left] {\begin{minipage}[lt]{17.15pt}\setlength\topsep{0pt}
$\displaystyle w_{1}$\\
\end{minipage}};
\draw (395,108) node   [align=left] {\begin{minipage}[lt]{17.15pt}\setlength\topsep{0pt}
$\displaystyle w_{3}$\\
\end{minipage}};
\draw (230,108) node   [align=left] {\begin{minipage}[lt]{17.15pt}\setlength\topsep{0pt}
$\displaystyle w_{2}$\\
\end{minipage}};

\end{tikzpicture}
	\end{figure} 
 \end{figure}

\begin{quote}

\mh The accessibility of this example is similar to the previous one: 
$w_2$ and $w_3$ have local access, but $w_1$ can access everything. 
 However, now $w_1$ is a $\cw$-world, so it sees the worlds $w_2$ and $w_3$ `through $C2^w$-glasses', which 
 means that some restrictions are imposed on a user in $w_1$. 

 From $w_1$, the value $T$ assigned to $p$ in $w_2$ is seen as  $T_0$ (a user in $w_1$ has no information about certification), and the value $ \nei $ of $p$ in $w_3$ is seen as $F_0$ (a $C2^w$-world maps everything below $T_0$ to $F_0$). 
 Therefore, according to the information available in $w_1$ (which includes the information in $w_2$ and $w_3$ as seen from $w_1$), 
 $p$ is just false, that is, $v_{w_1}(\Box p) = F_0$.  
\end{quote}
   

\end{example}

\begin{example} \ \label{ex.3}

\mm $I(w_1) = L6, \ I(w_2) = \cw , \  I(w_3) = \ns$ 

   
   $v_{w_1}(p)= \bo  $,  \ $ v_{w_1}(q)= \nei $

      \m  $   (v_{w_1}(p))^{L_{w_2}}= v_{w_2}(\Box p) = F_0  $

         $   (v_{w_1}(p))^{L_{w_3}} = v_{w_3}(\Box p) = F  $ 
 
 \m  $ (v_{w_1}(q))^{L_{w_2}}= v_{w_2}(\Box q) = F_0  $ 
  
     $  (v_{w_1}(q))^{L_{w_3}}= v_{w_3}(\Box q) = \nei   $

%

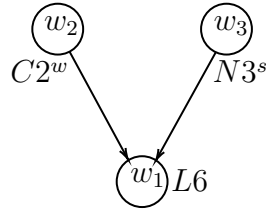
 \begin{figure}[H]
     \centering

\tikzset{every picture/.style={line width=0.75pt}} 

\begin{tikzpicture}[x=0.75pt,y=0.75pt,yscale=-0.51,xscale=0.51]

\draw   (362,93) .. controls (362,79.19) and (373.19,68) .. (387,68) .. controls (400.81,68) and (412,79.19) .. (412,93) .. controls (412,106.81) and (400.81,118) .. (387,118) .. controls (373.19,118) and (362,106.81) .. (362,93) -- cycle ;
\draw   (196,93) .. controls (196,79.19) and (207.19,68) .. (221,68) .. controls (234.81,68) and (246,79.19) .. (246,93) .. controls (246,106.81) and (234.81,118) .. (221,118) .. controls (207.19,118) and (196,106.81) .. (196,93) -- cycle ;
\draw   (279,243) .. controls (279,229.19) and (290.19,218) .. (304,218) .. controls (317.81,218) and (329,229.19) .. (329,243) .. controls (329,256.81) and (317.81,268) .. (304,268) .. controls (290.19,268) and (279,256.81) .. (279,243) -- cycle ;
\draw    (318.56,220.28) -- (376.59,115.53) ;
\draw [shift={(317.59,222.03)}, rotate = 298.99] [color={rgb, 255:red, 0; green, 0; blue, 0 }  ][line width=0.75]    (10.93,-3.29) .. controls (6.95,-1.4) and (3.31,-0.3) .. (0,0) .. controls (3.31,0.3) and (6.95,1.4) .. (10.93,3.29)   ;
\draw    (289.13,219.27) -- (232.59,115.03) ;
\draw [shift={(290.09,221.03)}, rotate = 241.52] [color={rgb, 255:red, 0; green, 0; blue, 0 }  ][line width=0.75]    (10.93,-3.29) .. controls (6.95,-1.4) and (3.31,-0.3) .. (0,0) .. controls (3.31,0.3) and (6.95,1.4) .. (10.93,3.29)   ;



\draw (197,134) node   [align=left] {\begin{minipage}[lt]{17.15pt}\setlength\topsep{0pt}
 $\cw$
\end{minipage}};
\draw (355,241.47) node   [align=left] {\begin{minipage}[lt]{17.15pt}\setlength\topsep{0pt}
 $L6$
\end{minipage}};
\draw (396,134) node   [align=left] {\begin{minipage}[lt]{17.15pt}\setlength\topsep{0pt}
 $\ns$
\end{minipage}};
\draw (315,256) node   [align=left] {\begin{minipage}[lt]{17.15pt}\setlength\topsep{0pt}
$\displaystyle w_{1}$\\
\end{minipage}};
\draw (397,106) node   [align=left] {\begin{minipage}[lt]{17.15pt}\setlength\topsep{0pt}
$\displaystyle w_{3}$\\
\end{minipage}};
\draw (230,106) node   [align=left] {\begin{minipage}[lt]{17.15pt}\setlength\topsep{0pt}
$\displaystyle w_{2}$\\
\end{minipage}};
\end{tikzpicture}

 \end{figure}

\begin{quote}
\mh 
This example intends to represent how different users (say,  with different privileges) see a database $L6$. 
The worlds $w_2$ and $w_3$ represent these users, thus the valuations in these worlds do not matter in this case and they do not access themselves.  

The values in the $L6$-world $w_1$ are seen `through $C2^w$- and $N3^s$-glasses'. 
 There is contradictory information about $p$ in the $L6$-world $w_1$, but  $w_2$  sees it as  
 just false $F_0$.  
The user represented by \( w_3 \) sees all definite information as certified; therefore, in \( w_3 \), \( p \) gets the value \( F \),   certified false.
 Since in this case sentences $\Box A$ are evaluated only w.r.t. how $w_1$ is seen from the other worlds,  the value of $ \Box p $ seen from a user in $w_2$ is just false $F_0$ but from   $w_3$ it is certified   false $F$. 
 Mutatis mutandis for $\Box q$.

\end{quote}

\end{example}

\begin{example} \ \label{ex.4} 

\m\mh In this example, assume that $R$ is transitive. 

 \m  $I(w_1) =I(w_4)=I(w_7)= L6 $

  $ I(w_2) = I(w_5) = I(w_8) = \lw  $ 

  $ I(w_3) = I(w_6) =I(w_9) = \bs $

\m $v_{w_1}(p)=T$, \  $v_{w_2}(p)=T_0$, \  $v_{w_3}(p)=\bo$. 
 
$v_{w_4}(p)=\nei $, \  $v_{w_5}(p)=T_0$,  \ $v_{w_6}(p)=\bo $.  

 $v_{w_7}(p)=T$, \ $v_{w_8}(p)=T_0$, \ $v_{w_9}(p)=T$. 
 
\m  $v_{w_1}(\Box p) = \bigwedge \{ T_0,T,\nei ,\bo  \}=F_0$,

$v_{w_4}(\Box p) = \bigwedge \{ T_0,T, \nei ,\bo  \}=F_0$,

$v_{w_7}(\Box p) = \bigwedge \{ T_0,T \}=T_0$. 

 
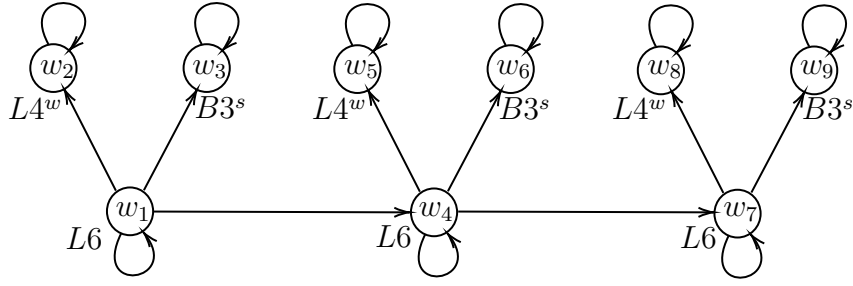
\begin{figure}[H]
     \centering
    \input{example4.tex}
 \end{figure}

\begin{quote}
This example aims to represent a possible evolution of the structure of Example \ref{ex.1} over time and can be interpreted as follows.

In the first stage, there is a contradiction in the world $w_3$: the value $\bo$ is assigned to $p$, and $\Box p$ is assigned the value $F_0$. 
This motivates a more cautious attitude in the central database, whose evolution is represented by the $L6$-worlds $w_1$, $w_4$, and $w_7$. 

 In the second stage,  $\nei$ is then assigned to $p$ in $w_4$, indicating a suspension of judgment, 
but  $\Box p$ still is assigned the value $F_0$.

 Finally, in the third stage, the problem is resolved in the $\bs$-database, and $p$ is assigned $T$ in both $w_7$ and $w_9$. $\Box p$ in $w_7$ is then assigned the value $T_0$.

\end{quote}

\end{example}

\section{Modalities in structures  based on \lbb\ and $L6$} \label{sec.modalities}

In this section, we connect the topics discussed in this paper with traditional investigations in modal logic. 
Specifically, we consider general validities within the proposed framework  of \mls s based on  the set 
\lbb\ of matrix logics  and explore the connection between modal formulas and the characterization of modal frames. 
To achieve this, we first need to define what a frame is in this new context. 
 Due to the presence of various logics within a many-logic structure, 
we are faced with two options for defining a frame: (i) treating a world within a frame as merely a point in a graph, or (ii) 
considering it as a point in a graph that also includes a reference to the logic operating in the world.
The option (i) corresponds to what is done in standard Kripke models, since there is only one underlying logic. In this case, 
a frame is defined as a pair $\pair{W,R}$. In our case, since different logics may be assigned to different worlds by the 
function $I$, a frame must contains this information, and so will be defined below as a triple $\pair{W, R, I}$.  
As discussed in \cite{RM24}, this approach is more intuitive and enables us to identify richer properties of frames.




We will now examine some positive and negative results with respect to the behavior of $\Box$ in structures based on \( L6 \) and \lbb. 
More precisely,  the modal axioms \textsf{(K)}, \textsf{(T)}, and \textsf{(4)} are, so to speak, `well-behaved' in the structures with the adequate accessibility relations, but the rule of necessitation does not hold. In order to see this in the present example, we will take advantage of the fact that all the logics in \lbb\ have a negation and an implication with suitable properties.

First we state some simple facts about the lattices being considered:




\begin{lemma}\label{implicationProperty}
Let \textbf{L} be a logic  in \lbb\ and $L$ the corresponding lattice. 
 Given   $a, b \in L$, we have that  $(a \tilde{\to} b) \notin \rm  D$ if, and only if, $a \in \rm D$ and 
    $b \notin \rm D$.
\end{lemma}

\begin{proof} Recall from Definition~\ref{def.matrices} the set of matrices considered in $\mathbb{L}$ and the definition of the set $D$ in each case. The result can be seen either by inspecting the truth tables or from the definition of the implication operator in the corresponding twist structures, see Definition~\ref{def.twist.geral}. 
\end{proof}

\begin{lemma}\label{upwardClose}
    Let \textbf{L} be a logic  in \lbb\ and $L$ the corresponding lattice. 
 Given   $a, b \in L$, we have that $a \in \rm D$ and $a \leq b$ implies $b \in\rm  D$.
        
\end{lemma}

\begin{proof}
It follows by inspecting the set of matrices considered in $\mathbb{L}$, see Definition~\ref{def.matrices}.
\end{proof}

\begin{proposition} \label{prop.behavior.modal} \ 
     \begin{enumerate}
      
      \item[] (i) Given a world $w$ and a formula $A$, the fact that $A$ holds in all worlds  accessible from $w$ does not imply that $\Box A$ holds in $w$. 
     
     \item[] (ii) The necessitation rule does not hold.

      \item[] (iii)  The modal axiom \textsf{(K)} $\Box (A \to B) \to (\Box A \to \Box B)$ 
      holds in all frames. 

 \end{enumerate}

\end{proposition}

\begin{proof} \ 

\m\mh For items (i) and (ii) consider the following structure:

\begin{figure}[H]
    \centering
\tikzset{every picture/.style={line width=0.75pt}} 
\begin{tikzpicture}[x=0.75pt,y=0.75pt,yscale=-0.45,xscale=0.45]
\draw   (97,1822.52) .. controls (97,1808.71) and (108.19,1797.52) .. (122,1797.52) .. controls (135.81,1797.52) and (147,1808.71) .. (147,1822.52) .. controls (147,1836.32) and (135.81,1847.52) .. (122,1847.52) .. controls (108.19,1847.52) and (97,1836.32) .. (97,1822.52) -- cycle ;
\draw   (271,1822.52) .. controls (271,1808.71) and (282.19,1797.52) .. (296,1797.52) .. controls (309.81,1797.52) and (321,1808.71) .. (321,1822.52) .. controls (321,1836.32) and (309.81,1847.52) .. (296,1847.52) .. controls (282.19,1847.52) and (271,1836.32) .. (271,1822.52) -- cycle ;
\draw    (147,1822.52) -- (269,1822.52) ;
\draw [shift={(271,1822.52)}, rotate = 180] [color={rgb, 255:red, 0; green, 0; blue, 0 }  ][line width=0.75]    (10.93,-3.29) .. controls (6.95,-1.4) and (3.31,-0.3) .. (0,0) .. controls (3.31,0.3) and (6.95,1.4) .. (10.93,3.29)   ;

\draw (101,1810) node [anchor=north west][inner sep=0.75pt]   [align=left] {$w_1$};
\draw (273,1810) node [anchor=north west][inner sep=0.75pt]   [align=left] {$w_2$};
\draw (355.67,1775) node   [align=left] {\begin{minipage}[lt]{68pt}\setlength\topsep{0pt}
$B3^w$
\end{minipage}};
\draw (109.67,1765) node [anchor=north west][inner sep=0.75pt]   [align=left] {$N3^w$};

\end{tikzpicture}
\end{figure}

\begin{multicols}{4}

$I(w_1)=\nw $ 

$ I(w_2)=\bw$

$R=\{ \pair{w_1, w_2}\}$ \\ 
    
 $v_{w_1}(p)=T_0$

$v_{w_1}(q)=T_0$ 

  $v_{w_2}(p)=\bo $ 

$v_{w_2}(q)=\bo$

\end{multicols}

\m\mh 
(i) Since $p$ holds in $w_2$,   $p$ holds in all worlds accessible from $w_1$. However, the value $\bo$ is seen as $F_0$ in $w_1$, 
 $(v_{w_2}(p))^{L_{w_1}}=F_0 $, so $v_{w_1}(\Box p)=F_0 $. Therefore, 
  $\Box p$ does not hold in $w_1$.

\mm\mh (ii) {To show that necessitation does not hold}, note that 
$p\to  (p\vee q)$ holds in every $w$ of every  structure $M$, since it is a valid formula in all matrix logics in $\mathbb{L}$. 
 Indeed, 
$v_{w_1} (p\to  (p\vee q))=T_0$ and  $v_{w_2}  (p\to  (p\vee q))=\bo $. 
However, 
since $v_{w_2} (p\to  (p\vee q))^{L_{w_1}}=F_0 $,  
  $v_{w_1} (\Box (p\to  (p\vee q)))=F_0$. Therefore  $\Box (p\to  (p\vee q))$
 does not hold in $w_1$.

\mmm\mh  (iii) 
 We show that for any world $w$ and any structure $M$,    
 $v_w(\Box (A\to B) \to (\Box A \to \Box B)) \in \rm D$, where  $\rm{D} = \{T, T_0, \bo \}$. 
{Suppose (i) $v_w(\Box (A\to B))\in \rm D$ and (ii) $v_w(\Box A) \in \rm D$.  
From (i), for every $w'\in W$ such that $wRw'$,$v_{w'}(A\to B)\in D$. Then, by Lemma \ref{implicationProperty}, $v_{w'}(A)\not \in D$ or $v_{w'}(B)\in D$. However, by (ii) $v_{w'}(A)$ must be in $ D$. Therefore, $v_{w'}(B)\in D$. 
Since $w'$ is any world in $W$ such that $wRw'$, $v_{w}(\Box B)\in D$.  }
\end{proof}


















Let us now focus on the modal axioms \textsf{(T)} and  \textsf{(4)}.  We start by defining a frame in a many-logic structure.   


\begin{definition}
    Let $M = \pair{W, R, I, v}$ be an \lbb-structure.  
    We say that $\mathcal{F} = \pair{W, R, I}$ is the underlying  $\mathbb{L}$-frame and $M$ is a structure with frame $\mathcal{F}$. 
\end{definition}

\begin{theorem}  \label{th.refl}
    Let $\mathcal{F}$ be an $\mathbb{L}$-frame. If $\mathcal{F}$ is reflexive,  then $\mathcal{F}$ satisfies \textsf{(T)}.
\end{theorem}

\begin{proof}
    Suppose that $\mathcal{F} = \pair{W,R,I}$ is an $\mathbb{L}$-frame and that it is reflexive. Now consider a world $w$ in a model $M = \pair{W,R,I, v}$ such that $w \nvDash \Box A \to A$. 
    From Lemma \ref{implicationProperty}, we obtain $v_w(\Box A) \in D$ and $v_w(A) \notin \rm  D$. Because $\mathcal{F}$ is reflexive, $w R w$ and so $v_w(\Box A) \leq v_w(A)$. Since the set $D$ is upward closed (i.e. from Lemma \ref{upwardClose}), we have the absurd conclusion that $v_w(A) \in D$. Therefore, $w \vDash \Box A \to A$.

    
    
\end{proof}

\begin{theorem}  \label{th.trans}
    Let $\mathcal{F}$ be an $\mathbb{L}$-frame. 
    If  $\mathcal{F}$ is transitive,  then $\mathcal{F}$ satisfies \textsf{(4)}. 
\end{theorem}

\begin{proof}
    Let $\mathcal{F} = \pair{W,R,I}$ be a transitive $\mathbb{L}$-frame and $M = \pair{W,R,I, v}$ be a model with underlying frame $\mathcal{F}$. Consider any $w \in W$ and any formula $A$. Because of transitivity, 
    $$S^2_w = \{u  \ : \  \exists w' (w R w' \land w' R u)\} \subseteq \{u  \ : \  w R u\} = S^1_w$$
    For any set $S\subseteq W$ and formula $B$, let $S(B, w)$ be the set $\{v_u(B)^{L_w}  \ : \  u \in S\}$. 
    Then
    $$S(S^2_w, w) \subseteq  S(S^1_w, w)$$
    Note that $v_w(\Box B) = \bigwedge S(S^1_w, w)$ and $v_w({\Box}{\Box}B) = \bigwedge S(S^2_w, w)$. Hence 
    $$v_w({\Box}A) \leq v_w({\Box}{\Box}A)$$
    Consequently, from Lemma \ref{implicationProperty} and Lemma \ref{upwardClose}, $w \vDash {\Box}A \so {\Box}{\Box}A$. Since $w$ is any world, we obtain that the frame indeed satisfies axiom \textsf{(4)}.

\end{proof}

Regarding axioms \textsf{(5)}, ${\Diamond} A \to {\Box } {\Diamond} A$, \textsf{(B)}, $A \to \Box \Diamond A$, and \textsf{(D)}, ${\Box }A \to \Diamond A$, their behavior in our frames will depend further on how $\Diamond A$ is defined. We have at least four  alternatives to define it. First, it may be  defined by taking the supremum  
$\bigvee$ instead of the infimum  $\bigwedge$ in the Definition \ref{def.clause.box}:
\begin{equation}\label{diamond-def-down}
v_{{w}}(\Diamond  A)=\bigvee\limits_{L_{{w}}}\{(v_{w'}(A))^{L_{{w}}}  \ : \  w' \in W \land {w} R w'  \}    
\end{equation}
Second, observe that a classical $\sneg$ negation can be defined in all logics in \lbb\ as
$\sneg A \defi A\to \bot$, since these logics include $\bot$\footnote{For those lattices that do not have values $F$ and $T$, $\circ  A $ is assigned to $F_0$ for every $ A $, i.e. the least value in the corresponding lattice.} and satisfy the schema $A\lor (A\to B)$.
Thus, an alternative definition of ${\Diamond}$ is ${\Diamond}A \defi \sneg {\Box} \sneg A$.
A third option is ${\Diamond}A \defi \neg {\Box} \neg A$, where $\neg$ denotes the negation specific to each logic.
Naturally, these alternatives are not expected to produce identical results.

Observe that defining $\Diamond$ in terms of classical negation is rather irregular.
In particular, regardless of the lattice in $\LAT $ under consideration, the value of a formula $\sneg A$ is always confined to the four-element set $\{T, T_{0}, F_0, F\}$. To illustrate, suppose that every world in a model operates in $\lj$ and that $A$ is assigned the value $b$ at every world. Then $\sneg  A$ evaluates to $F$ at every world; consequently, $\sneg \Box \sneg A$ evaluates to $T$ at every world. This runs counter to the usual intuition that, whenever a formula has a constant value across all worlds, its possibility should preserve that constancy (i.e. it should return that same value).

We may then compare the definition in \cref{diamond-def-down} with the clause $\lnot \Box \lnot$. Consider an $\fdetobot$ world $w$ that accesses only a $\ktto$ world at which $A$ has value $n$. In this situation,
$$v_w(\Diamond A)=F_0  \quad\text{while}\quad  v_w(\lnot \Box \lnot A)=T_0$$
Thus, under \cref{diamond-def-down} we do not, in general, obtain the usual duality between $\Box$ and $\Diamond$. The advantage in this case is that the equivalence between $\Diamond$ and $\lnot \Box \lnot$ will hold in the cases where every world operates in only one logic. 

As our fourth alternative, we replace the down-interpretation of the relativized values with the corresponding up-interpretation. This, as we will see next, fully restores the duality between $\Box$ and $\Diamond$.

\begin{definition} (The up-interpretation) \label{def.up.int} 	
\m	\noindent Let  $L$ be a complete lattice and $L'$ a dc-sublattice of $L$. 
	The {\it up-interpretation} of value $x \in L$  in  $L'$ is defined as follows:
\begin{center}
$ x^{L'}_{up} = \bigwedge_{L'} \{y \in L' \ : \  y \geq x \} $.
\end{center}
\end{definition}

\noindent
Note that, if $\{y \in L'  \ : \   y \geq x \} = \emptyset$, then $x^{L'}_{up}$ is the greatest value in $L'$. Moreover, if $x \in L'$ then $ x^{L'}_{up} = x$.

\begin{definition} (Semantic clause for $\Diamond$) \ \label{def.clause.diamond}
\begin{center}
$v_{{w}}(\Diamond  A)=\bigvee\limits_{L_{{w}}}\{(v_{w'}(A))^{L_{{w}}}_{up}  \ : \  w' \in W \land {w} R w'  \}$. 
\label{clause.box}
\end{center}
\end{definition}

Under this definition, we obtain that $\Diamond$ and $\lnot \Box \lnot$ are equivalent:

\begin{lemma} \label{props-neg}
 Let \textbf{L} be a logic  in \lbb\ and $L$ the corresponding lattice. Let $\neg_L$ be the negation in \textbf{L}.
 Given $A \subseteq L$ and $x, y \in L$, we have that
 \begin{enumerate}
     \item $\neg_L \bigwedge_L A = \bigvee_L \{\neg_L\, x \ : \  x \in A\}$
     \item $\neg_L \bigvee_L A = \bigwedge_L \{\neg_L\, x \ : \  x \in A\}$
     \item $x \leq \neg_L\, y$ if and only if $\neg_L\, x \geq y$.
     \item $\neg_L:L \to L$ is a bijection.
     \item $\neg_L\, x = \neg_{L6}\, x$, where $\neg_{L6}$ is the negation in $L6$.
 \end{enumerate}
 \end{lemma}

 \begin{proof}
     The result is not general for every lattice, but it is for the corresponding lattices being considered in \lbb. For item 5, recall  Definition~\ref{def.twist.geral}. We leave the (easy) proof to the reader.
 \end{proof}

\begin{theorem}
    Every world $w$ in any \lbb-structure is such that $v_w(\Diamond A) = v_w(\lnot \Box \lnot A)$.
\end{theorem}

\begin{proof}
    For each $w \in W$ let $L_w$ be the corresponding lattice. Given $w \in W$, we calculate the value of $v_w(\lnot \Box \lnot A)$ using the previous lemma:
    \begin{align*}
        v_w(\lnot \Box \lnot A ) &= \neg_{L_w} \bigwedge_{L_w} \{{(v_{w'}(\lnot A ))}^{L_w}_{down} \ : \  w R w'\} \ \mbox{$v_w$ is a valuation}\\
        &= \neg_{L_w} \bigwedge_{L_w} \{\bigvee_{L_w}\{x \in L_w \ : \  x \leq v_{w'}(\lnot A )\} \ : \  w R w'\} \ \mbox{by definition of $x^{L_w}_{down}$}\\
        &= \bigvee_{L_w} \{\neg_{L_w} \bigvee_{L_w}\{x \in L_w \ : \  x \leq v_{w'}(\lnot A )\} \ : \  w R w'\} \ \mbox{by Lemma~\ref{props-neg}(1)}\\
        &= \bigvee_{L_w} \{\bigwedge_{L_w}\{\neg_{L_w}\, x \ : \  x \in L_w, \  x \leq v_{w'}(\lnot A )\} \ : \  w R w'\} \ \mbox{by Lemma~\ref{props-neg}(2)}\\
        &= \bigvee_{L_w} \{\bigwedge_{L_w}\{\neg_{L_w}\, x \ : \  x \in L_w, \  x \leq \neg_{L_{w'}} v_{w'}(A )\} \ : \  w R w'\} \ \mbox{$v_{w'}$ is a valuation}\\
        &= \bigvee_{L_w} \{\bigwedge_{L_w}\{\neg_{L_w}\, x \ : \  x \in L_w, \ x \leq \neg_{L6}\, v_{w'}(A )\} \ : \  w R w'\} \ \mbox{by Lemma~\ref{props-neg}(5)}\\
        &= \bigvee_{L_w} \{\bigwedge_{L_w}\{\neg_{L_w}\, x \ : \  x \in L_w, \ \neg_{L6}\, x \geq v_{w'}(A )\} \ : \  w R w'\} \ \mbox{by Lemma~\ref{props-neg}(3)}\\
        &= \bigvee_{L_w} \{\bigwedge_{L_w}\{\neg_{L_w}\, x \ : \  x \in L_w, \ \neg_{L_w}\, x \geq v_{w'}(A )\} \ : \  w R w'\} \ \mbox{by Lemma~\ref{props-neg}(5)}\\
        &= \bigvee_{L_w} \{\bigwedge_{L_w}\{x \in L_w \ : \  x \geq v_{w'}(A )\} \ : \  w R w'\} \ \mbox{by Lemma~\ref{props-neg}(4)}\\
        &= \bigvee_{L_w} \{{(v_{w'}(A ))}^{L_w}_{up} \ : \  w R w'\} \ \mbox{by definition of $x^{L_w}_{up}$}\\
        &= v_{w}(\Diamond A ) \ \mbox{by Definition~\ref{def.clause.diamond}}.
    \end{align*}
\end{proof}

Again using Definition~\ref{def.clause.diamond}, we may connect Eucliden frames with modal formulas. First, let us observe that euclidean frames are not characterized by axiom \textsf{(5)}. Suppose the Euclidean frame composed of 3 worlds $w_1$, $w_2$ and $w_3$ each accessing each other; the lattices of $w_1$, $w_2$ and $w_3$ are $\fdetobot$, $\ktto$ and $\lpto$ respectively. Let the value of $A$ be $F_0$ in $w_1$, $F_0$ in $w_2$ and $b$ in $w_3$. We thus obtain $v_{w_1}(\Diamond A) = v_{w_3}(\Diamond A) = b$, $v_{w_2}(\Diamond A) = F_0$ and, consequently, $v_{w_1}(\Box \Diamond A) = F_0$. The value of axiom \textsf{(5)} is $F_0$ in $w_1$ and so it fails in an Euclidean frame.

Now, following the strategy of \cite{freiremartins-coniglio}, we modify formula \textsf{(5)} in order to obtain the desired characterization. In that work, the authors use the formula $\Diamond \circ A  \to \Box \Diamond \circ A $ to characterize Euclidean frames for a range of four-valued Boolean lattices. Here,
an analogous move works (for Euclidean frames, and also for other frame properties).
Recalling that the values of $\sneg A $ are confined to the set $\{T, T_{0}, F_0, F\}$,
we employ the formula $\Diamond \sneg A  \to \Box \Diamond \sneg A $ to characterize Euclidean frames.\footnote{The appeal to classical negation in the characterization of \textsf{(5)} suggests that the Euclidean property is more closely connected to classical valuations than the other properties. Further work is needed to clarify the nature of this connection and what, exactly, it amounts to.}

\begin{theorem}\label{th.euclidean}
    Let $\mathcal{F}$ be an $\mathbb{L}$-frame. 
    $\mathcal{F}$ is Euclidean if and only if $\mathcal{F}$ satisfies $\Diamond \sneg A  \to \Box \Diamond \sneg A $ for every $A $. 
\end{theorem}

\begin{proof}
    Since $\sneg A $ can only receive values $T$, $T_0$, $F_0$ or $F$ and every lattice in $\LAT $ has values $\{T, F\}$, $\{T_0, F_0\}$ or $\{T, F, T_0, F_0\}$, any formula of the form $*_1 *_2 \ldots *_n \sneg A $, where $*_i$ is either $\Box$ or $\Diamond$, receive value $T$, $T_0$, $F$ or $F_0$ in every world. This is easily proved by induction.

    Now, suppose $M$ is an Euclidean structure such that there is a world $w$ with lattice $L$ and a formula $A $ for which
    $$w \nvDash \Diamond \sneg A  \to \Box \Diamond \sneg A $$
    Consequently, $w \vDash \Diamond \sneg A $ and $w \nvDash \Box \Diamond \sneg A $ from \cref{implicationProperty}. 
    So there is a world $w'$ such that $w R w'$ and $v_{w'}(\sneg A ) \geq T_0$ and there is $w''$ such that $w R w''$ and $v_{w''}(\Diamond \sneg A ) \leq F_0$. Now, since the frame is Euclidean, $w'' R w'$. But this implies the contradiction $v_{w''}(\Diamond \sneg A ) \geq T_0$.

    Now suppose $\mathcal{F}$ is a non Euclidean frame. There are $w_1$, $w_2$ and $w_3$ in $\mathcal{F}$ such that $w_1 R w_2$, $w_1 R w_3$, but $w_2 \not R w_3$. Let us use the notation $T'$ to refer to $T_0$ or $T$ and $F'$ to refer to $F_0$ or $F$. We obtain the model $M$ by assigning in $\mathcal{F}$ the value $F'$ to $A$ in every $w'$ such that $w_2 R w'$ and $T'$ to $A$ in $w_3$. This is a consistent attribution of values because $w_2 \not R w_3$. Consequently, $v_{w_1}(\Diamond \sneg A) \geq T_0$ and $v_{w_2}(\Diamond \sneg A) \leq F_0$. From this, we obtain that $v_{w_1}(\Box \Diamond \sneg A) \leq F_0$ and so that the frame $\mathcal{F}$ does not satisfy the formula $\Diamond \sneg A  \to \Box \Diamond \sneg A $.
\end{proof}

The attentive reader may have noticed that Theorems \ref{th.refl} and \ref{th.trans}
do not strongly depend on particular aspects of the base lattice $L6$ nor on the logics in $\mathbb{L}$.
Indeed, we relied only on Lemma \ref{implicationProperty} and Lemma \ref{upwardClose}
for each lattice under consideration. 
This suggests a more general study on the characterization of frames in many-logic modal systems, which will be explored in future work.

\section{Final remarks}  \label{sec.final}

To the best of our knowledge, the logic \fde, together with the respective lattice \lw\ and the intuitive interpretation proposed and investigated by Belnap 
and Dunn,  
  is   the first information-based logic to appear in the literature. 
  \fde\ and its extensions are widely investigated as logics  capable of
expressing the deductive behavior of information (cf.~\citet[p.~928]{wansing.answer.dubois}). 
All the logics investigated here are extensions of \fde. We adopted versions of \fde, \textit{K3}, and 
\textit{LP} with an implication  validating the deduction theorem and modus ponens. 
 This choice was made for the obvious reason that 
   an implication with these features provides more expressive power than the implication definable in the original version of these logics as $\neg A\lor B$. 
 We see   the intuitive interpretations of the formal systems and the respective lattices here investigated as a further 
development of Belnap-Dunn's proposal. 
However, the question of  how these logics could be effectively applied, for instance, in database management systems, falls outside the scope of this paper. 

That  paraconsistent logics are suitable for
expressing the deductive behavior of information   
  is nowadays widely established in the literature\footnote{See e.g. \cite{Bertossietal2005, Blair&Subrahmanian1989, deAmoetal2002,  wans.93}}.  Paracomplete logics in general have been accepted more easily, and earlier, than paraconsistent logics for dealing with 
  information, but the idea that positive and negative information must be treated  on a par indeed leads us to a
logic both paracomplete and paraconsistent. 
On the other hand, some circumstances  may require only paraconsistent, only paracomplete, or even  classical 
 scenarios. We proposed not only  accounts of such circumstances, based on six, four-, three-, and two-valued logics, 
 but also  how these different logics can `communicate' with each other. 
 Moreover, a feature of the many-logic modal structures   here proposed that fits with the concept of an information-based logic is their ability to represent information states that evolve in various ways, depending on how new information is acquired over time.

 As far as  further work is concerned, the following topics seem worth investigating.
The first is to define an operation of merging between databases. 
The idea is that given two structures $M$ and $M'$ and a world $w$ that belongs to both
$M$ and $M'$, the merging yields a structure $M''$ in which $w$ contains
the formulas $p$, $\neg p$, and $\con p$ that hold in $w$ in $M$ and in $M'$, with the proviso
that the operation cannot yield trivial worlds.
The second is to investigate an  approach to the evolution of databases over time that allow 
the definition of temporal modalities  such as
past, future,  next, until, and since. 
Finally, measures of inconsistency  and data certification can be defined, based on 
the number on variables $p$ such that $p$ and $\neg p$ hold in given a world, or in a given structure, and also how this number increases or decreases over time. Analogously, measures of data certification can be defined based on the number of formulas $\con p$.

\bigskip

\noindent {\bf Compliance with Ethical Standards.} This article does not contain any studies with human participants or animals performed by any of the authors.

\medskip

\noindent  {\bf Funding.} This work was partially supported by CIDMA (https://ror.org/05pm2mw36)
under the Portuguese Foundation for Science and Technology 
(FCT, https://ror.org/00snfqn58), Grants UID/04106/2025 (https://doi.org/10.54499/UID/04106/2025)\\
and UID/PRR/04106/2025; and European Funds through SACCCT- IC\&DT - Sistema de Apoio \`a Cria\c c\~ao de Conhecimento Cient\'ifico e Tecnol\'ogico, as part of COMPETE2030, within the project BANKSY with reference number 15253.

The second author acknowledges support from the National Council for Scientific and Technological
Development (CNPq), research grant 307889/2025-4.

The second and the third author acknowledge support from the Sao Paulo Research Foundation (FAPESP, Brazil), thematic project {\em Rationality, logic and probability -- RatioLog}, grant  2020/16353-3.

The third author acknowledges support from the National Council for Scientific and Technological
Development (CNPq), research grant 309830/2023-0.

\bibliographystyle{plainnat}
\bibliography{refs.main}

\end{document}